\RequirePackage{fix-cm}
\documentclass{svjour3}                     
\smartqed  
\usepackage{graphicx}
\usepackage{enumerate}
\usepackage{amsfonts,amssymb} 
\usepackage{amsmath} 
\begin{document}

\title{Numerical algorithm for the space-time fractional Fokker-Planck system with two internal states\thanks{This work was supported by the National Natural Science Foundation of China under Grant No. 11671182 and the Fundamental Research Funds for the Central Universities under Grant No. lzujbky-2018-ot03.}
}
\titlerunning{Numerical algorithm for the space-time fractional Fokker-Planck equations}

\author{Daxin Nie         \and
        Jing Sun \and
        Weihua Deng 
}


\institute{Daxin Nie \at
              School of Mathematics and Statistics, Gansu Key Laboratory of Applied Mathematics and Complex Systems, Lanzhou University, Lanzhou 730000, P.R. China \\
              \email{ndx1993@163.com}           
           \and
           Jing Sun \at
              School of Mathematics and Statistics, Gansu Key Laboratory of Applied Mathematics and Complex Systems, Lanzhou University, Lanzhou 730000, P.R. China\\
           \email{sunj18@lzu.edu.cn}
           \and
           Weihua Deng\at
           School of Mathematics and Statistics, Gansu Key Laboratory of Applied Mathematics and Complex Systems, Lanzhou University, Lanzhou 730000, P.R. China\\
           \email{dengwh@lzu.edu.cn}
}

\date{Received: date / Accepted: date}

\maketitle

\begin{abstract}
The fractional Fokker-Planck system with multiple internal states is derived in [Xu and Deng, Math. Model. Nat. Phenom., $\mathbf{13}$, 10 (2018)], where the space derivative is Laplace operator. If the jump length distribution of the particles is power law instead of Gaussian, the space derivative should be replaced with fractional Laplacian.
This paper focuses on solving the two state Fokker-Planck system with fractional Laplacian.
We first provide a priori estimate for this system under different regularity assumptions on the initial data. Then we use $L_1$ scheme to discretize the time fractional derivatives and finite element method to approximate the fractional Laplacian operators. Furthermore, we give the error estimates for the space semidiscrete and fully discrete schemes without any assumption on regularity of solutions. Finally, the effectiveness of the designed scheme is verified by numerical experiments.

\keywords{Fractional Fokker-Planck system \and multiple internal states \and Riemann-Liouville fractional derivative \and fractional Laplacian \and $L_1$ scheme \and finite element method}
\end{abstract}
\section{Introduction}

Anomalous diffusion phenomena are widespread in the nature world \cite{Klafter2005}. Important progresses for modelling these phenomena have been made both microscopically by stochastic processes and macroscopically by partial differential equations (PDEs) \cite{Deng2019}. Generally, the PDEs govern the probability density function (PDF) of some particular statistical observables, say, position, functional, first exit time, etc.  The fractional Fokker-Planck equation models the PDF of the position of the particles \cite{Barkai2001,Barkai2000}.  So far, there are many numerical methods for solving FFPE, such as finite difference method, finite element method, and even the stochastic methods \cite{Deng2009,Deng2007,Heinsalu2006,Meerschaert2006,Sun2019}.

Anomalous diffusions with multiple internal states not only are often observed natural phenomena but also some challenge problems, e.g., smart animal searching for food, can be easily treated by taking them as a problem with multiple internal states. Recently, multiple-internal-state L\'evy walk and CTRW with independent waiting times and jump lengths are carefully discussed and the PDEs governing the PDF of some statistical observables are derived \cite{Xu2018,Xu2018-2}. In the CTRW model if the distributions of the jump lengths are power law instead of Gaussian, then the corresponding PDEs involve fractional Laplacian.
In this paper, we provide and analyze the numerical scheme for the following fractional Fokker-Planck system (FFPS) with two internal states \cite{Xu2018} and the appropriate boundary condition is specified \cite{Deng2019}, i.e.,
\begin{equation}\label{equmatrixequ}
\left \{
\begin{aligned}
&\mathbf{M}^T\frac{\partial }{\partial t} \mathbf{G}=(\mathbf{M}^T-\mathbf{I}){\rm diag}(~_0D^{1-\alpha_1}_t,~_0D^{1-\alpha_2}_t)\mathbf{G}\\
&\quad\quad\quad\quad\quad +\mathbf{M}^T{\rm diag}(-\,_0D^{1-\alpha_1}_t(-\Delta)^{s_1},-\,_0D^{1-\alpha_2}_t(-\Delta)^{s_2})\mathbf{G} \ \ \,{\rm in}\ \Omega,\ t\in(0,T],\\
&\mathbf{G}(\cdot,0)=\mathbf{G}_0 \quad\quad\quad\quad\quad\quad\quad\quad\quad\quad\quad\quad\quad\quad\quad\quad\quad\quad\qquad\qquad {\rm in}\ \Omega,\\
&\mathbf{G}=0 \quad\quad\quad\quad\quad\quad\quad\quad\quad\quad\quad\quad\quad\quad\quad\quad\quad\quad\quad\quad\qquad\qquad\ \ \, {\rm in}\ \Omega^c,\ t\in[0,T],
\end{aligned}
\right .
\end{equation}
where $\Omega$ denotes a bounded convex polygonal domain in $\mathbb{R}^n$ $(n=1,2,3)$; $\Omega^c$ means the complementary set of $\Omega$ in $\mathbb{R}^n$; $\mathbf{M}$ is the transition matrix of a Markov chain, being a $2\times 2$ invertible matrix here; $\mathbf{M}^T$ means the transpose of $\mathbf{M}$; $\mathbf{G}=[G_1,G_2]^T$ denotes the solution of the system \eqref{equmatrixequ}; $\mathbf{G}_0=[G_{1,0},G_{2,0}]^T$ is the initial value; $\mathbf{I}$ is an identity matrix; `diag' denotes a diagonal matrix formed from its vector argument;  $\,_0D^{1-\alpha_i}_t$ $(i=1,2)$ are the Riemann-Liouville fractional derivatives defined by \cite{Podlubny1999}
\begin{equation}
_{0}D^{1-\alpha_i}_tG=\frac{1}{\Gamma(\alpha_{i})}\frac{\partial}{\partial t}\int^t_{0}(t-\xi)^{\alpha_i-1}G(\xi)d\xi, ~\alpha_{i}\in(0,1);
\end{equation}
and $(-\Delta)^{s_i}$ $(i=1,2)$ are the fractional Laplacians given as
\begin{equation*}
    (-\Delta)^{s_i}u(x)=c_{n,s_i}{\rm P.V.}\int_{\mathbb{R}^n}\frac{u(x)-u(y)}{|x-y|^{n+2s_i}}dy,\quad s_i\in (0,1),
\end{equation*}
where  $c_{n,s_i}=\frac{2^{2s_i}s_i\Gamma(n/2+s_i)}{\pi^{n/2}\Gamma(1-s_i)}$ and ${\rm P.V.}$ denotes the principal value integral. Without loss of generality, we set $s_1\leq s_2$ in this  paper.

In some sense, the system \eqref{equmatrixequ} can be seen as the extension of the model
\begin{equation}\label{tramod}
\left\{
\begin{aligned}
&\frac{\partial }{\partial t} G=-\,_0D^{1-\alpha}_t(-\Delta)^{s}G \ \ \ \quad\quad\quad\,\,\,{\rm in}\ \Omega,\ t\in(0,T],\\
&G(\cdot,0)=G_0 \quad\quad\quad\quad\quad\quad\quad\quad\quad\quad {\rm in}\ \Omega,\\
&G=0 \quad\quad\quad\quad\quad\quad\quad\quad\quad\quad\quad\quad\ \ \, {\rm in}\ \Omega^c,\ t\in[0,T].
\end{aligned}
\right.
\end{equation}
It is well known that Eq. \eqref{tramod} has a wide range of practical applications, and there are also some discussions on its regularity and numerical issues \cite{Acosta20171,Acosta20172,Acosta20173,Acosta20174}; in particular, \cite{Acosta20174} provides an optimal spatial convergence rates when $s\in (1/2,1)$.
Compared with \eqref{tramod}, the solutions of the system \eqref{equmatrixequ} are coupled with each other and two different space fractional derivatives bring about a huge challenge on the priori estimates of the solutions. Here, we provide a priori estimate for the system \eqref{equmatrixequ} with $G_1(0),\,G_2(0)\in L^2(\Omega)$ (see Theorem \ref{thmhomoregularitynonsmooth}) and discuss the regularity of the system \eqref{equmatrixequ} detailedly with $s_1,s_2<1/2$ under different regularity assumptions on initial data (see Theorems \ref{thmhomoregularitysmooth1} and \ref{thmhomoregularitysmooth2}). Then we use the finite element method to discretize the fractional Laplacians and provide error analysis for spatial semidiscrete scheme. Lastly, we use $L_1$ scheme to discretize the time fractional derivatives and get the first order accuracy without any assumption on the regularity of the solutions. Besides, the proof ideas used in this paper can also be applied to \eqref{tramod} and an optimal spatial convergence rates can be got for $s\in (0,1)$ rather than $s\in (1/2,1)$.

The paper is organized as follows. In Section \ref{Sec1}, we first introduce the notations and then focus on  the Sobolev regularity of the solutions for the system \eqref{equmatrixequ} under different regularity assumptions on initial data. In Section  \ref{Sec2}, we do the space discretizations by the finite element method and provide the error estimates for the semidiscrete scheme. In Section \ref{Sec3}, we use the $L_1$ scheme to discretize the time fractional derivatives and provide error estimates for the fully discrete scheme. In Section \ref{Sec4}, we confirm the theoretically predicted convergence orders by numerical examples. Finally, we conclude the paper with some discussions. Throughout this paper, $C$ denotes a generic positive constant, whose value may differ at each occurrence, and $\varepsilon>0$ is an arbitrary small constant.

\section{Regularity of the solution} \label{Sec1}
In this section, we focus on the regularity of the system (\ref{equmatrixequ}).

\subsection{Preliminaries} Here we make some preparations. Denote $G_1(t)$, $G_2(t)$ as the functions $G_1(\cdot,t)$, $G_2(\cdot,t)$ respectively,  use the notation `$\tilde{~}$' for taking Laplace transform, and introduce $\|\cdot \|_{X\rightarrow Y}$ as the operator norm from $X$ to $Y$, where $X$, $Y$ are Banach spaces. Furthermore, for $\kappa>0$ and $\pi/2<\theta<\pi$, we denote sectors $\Sigma_{\theta}$ and $\Sigma_{\theta,\kappa}$ as
    \begin{equation*}
        \begin{aligned}
        \Sigma_{\theta}=\{z\in\mathbb{C}:z\neq 0,|\arg z|\leq \theta\},\ \Sigma_{\theta,\kappa}=\{z\in\mathbb{C}:|z|\geq\kappa,|\arg z|\leq \theta\},
        \end{aligned}
    \end{equation*}
 and define the contour $\Gamma_{\theta,\kappa}$ by
    \begin{equation*}
    \Gamma_{\theta,\kappa}=\{r e^{-\mathbf{i}\theta}: r\geq \kappa\}\cup\{\kappa e^{\mathbf{i}\psi}: |\psi|\leq \theta\}\cup\{r e^{\mathbf{i}\theta}: r\geq \kappa\},
    \end{equation*}
    where the circular arc is oriented counterclockwise, the two rays are oriented with an increasing imaginary part, and $\mathbf{i}$ denotes the imaginary unit. For convenience, in the following we denote $\Gamma_\theta=\Gamma_{\theta,0}$ and $A_i$ as the fractional Laplacian $(-\Delta)^{s_i}$ $(i=1,2)$ with homogeneous Dirichlet boundary condition. 

Then we recall some fractional Sobolev spaces \cite{Acosta20171,Acosta20172,Acosta20174,Di2012}. Let $\Omega\subset \mathbb{R}^n$ $(n=1,2,3)$ be an open set and $s\in(0,1)$. Then the fractional Sobolev space $ H^s(\Omega) $ can be defined by
\begin{equation*}
    H^s(\Omega)=\left\{w\in L^2(\Omega):|w|_{H^s(\Omega)}=\left(\int\!\int_{\Omega^2}\frac{|w(x)-w(y)|^2}{|x-y|^{n+2s}}dxdy\right)^{1/2}<\infty\right\}
\end{equation*}
with the norm $\|\cdot\|_{H^{s}(\Omega)}=\|\cdot\|_{L^2(\Omega)}+|\cdot|_{H^{s}(\Omega)}$, which constitutes a Hilbert space. As for $s>1$ and $s\notin \mathbb{N}$, the fractional Sobolev space $ H^s(\Omega) $ can be defined as
\begin{equation*}
    H^{s}(\Omega)=\{w\in H^{\lfloor s\rfloor }(\Omega):|D^{\alpha}w|_{H^{\sigma}(\Omega)}<\infty\ {\rm for\ all}\ \alpha\ s.t.\ |\alpha|=\lfloor s\rfloor\},
\end{equation*}
where $\sigma=s-\lfloor s\rfloor$ and $\lfloor s\rfloor$ means the biggest integer not larger than $s$. Another space we use is composed of functions in $H^s(\mathbb{R}^n)$ with support in $\bar{\Omega}$, i.e.,
\begin{equation*}
    \hat{H}^s(\Omega)=\{w\in H^s(\mathbb{R}^n):{\bf supp} \ w\subset \bar{\Omega}\},
\end{equation*}
whose inner product can be defined as the bilinear form
\begin{equation}\label{equdefHsOmega}
    \langle u,w\rangle_s:=c_{n,s}\int\int_{(\mathbb{R}^n\times\mathbb{R}^n)\backslash(\Omega^c\times\Omega^c)}\frac{(u(x)-u(y))(w(x)-w(y))}{|x-y|^{n+2s}}dydx.
\end{equation}

\begin{remark}
 According to \cite{Acosta20174}, the norm induced by \eqref{equdefHsOmega} is a multiple of the $H^{s}(\mathbb{R}^n)$-seminorm, which is equivalent to the full $H^{s}(\mathbb{R}^n)$-norm on this space because of the fractional Poincar\'{e}-type inequality \cite{Di2012}. Moreover, from \cite{Acosta20171}, $\hat{H}^s(\Omega)$ coincides with $H^s(\Omega)$ when $s<1/2$.
\end{remark}

 Next we recall the properties  and elliptic regularity of the fractional Laplacian. Reference \cite{Acosta20174} claims that $(-\Delta)^s:H^l(\mathbb{R}^n)\rightarrow H^{l-2s}(\mathbb{R}^n)$ is a bounded and invertible operator. Besides,  Ref. \cite{Grubb2015} proposes the regularity of the following problem
\begin{equation}\label{equdiripro}
    \left \{\begin{aligned}
    (-\Delta)^su&=g\quad {\rm in}\ \Omega,\\
    u&=0\quad {\rm in}\ \Omega^c,
    \end{aligned}\right.
\end{equation}
and the main results are described as
\begin{theorem}[\cite{Grubb2015}]\label{thmregdiri}
    Let $\Omega \subset \mathbb{R}^n$ be a bounded domain with smooth boundary, $g\in H^r(\Omega)$ for some $r\geq -s$ and consider $u\in \hat{H}^s(\Omega)$ as the solution of the Dirichlet problem \eqref{equdiripro}.
Then, there exists a constant $C$ such that
    \begin{equation*}
        |u|_{H^{s+\gamma}(\mathbb{R}^n)}\leq C\|g\|_{H^{r}(\Omega)},
    \end{equation*}
where $\gamma=\min(s+r,1/2-\epsilon)$ with $\epsilon>0$ arbitrarily small.
\end{theorem}

\subsection{A priori estimate of the solution}
According to the property of the transition matrix of a Markov chain \cite{Xu2018}, the matrix $\mathbf{M}$ can be denoted as
\begin{equation*}
    \mathbf{M}=\left [\begin{matrix}
    m&1-m\\
    1-m&m
    \end{matrix}\right ],\quad m\in[0,1/2)\cup(1/2,1],
\end{equation*}
and the fact that matrix $\mathbf{M}$ is invertible leads to
\begin{equation*}
(\mathbf{M}^T)^{-1}=\left [\begin{matrix}
\frac{m}{2m-1}&\frac{m-1}{2m-1}\\
\frac{m-1}{2m-1}&\frac{m}{2m-1}
\end{matrix}\right ].
\end{equation*}
So the system \eqref{equmatrixequ} can be rewritten as
\begin{equation}\label{equrqtosol}
    \left \{
    \begin{aligned}
    &\frac{\partial G_1}{\partial t}+a~_0D^{1-\alpha_1}_tG_1+~_0D^{1-\alpha_1}_tA_1 G_1=a~_0D^{1-\alpha_2}_tG_2\quad\quad\quad\,\, {\rm in}\ \Omega,\ t\in(0,T],\\
    &\frac{\partial G_2}{\partial t}+a~_0D^{1-\alpha_2}_tG_2+~_0D^{1-\alpha_2}_tA_2 G_2=a~_0D^{1-\alpha_1}_tG_1 \quad\quad\,\,\quad {\rm in}\ \Omega,\ t\in(0,T],\\
    &\mathbf{G}(\cdot,0)=\mathbf{G}_0 \,\,\quad\quad\quad\quad\quad\quad\quad\quad\quad\quad\,\,\quad\quad\quad\quad\quad\quad\quad\quad\quad\quad {\rm in}\ \Omega,\\
    &\mathbf{G}=0 \,\,\,\quad\quad\quad\quad\quad\quad\quad\quad\quad\quad\quad\quad\quad\quad\quad\quad\quad\quad\quad\quad\quad\quad\quad {\rm in}\ \Omega^c,\ t\in[0,T],
    \end{aligned}
    \right .
\end{equation}
where $ a=\frac{1-m}{2m-1} $, $m\in[0,1/2)\cup(1/2,1]$.

Taking the Laplace transforms for the first two equations of the system \eqref{equrqtosol} and using the identity $\widetilde{~_0D^{\alpha}_tu}(z)=z^\alpha\tilde{u}(z)$ \cite{Podlubny1999}, we have
\begin{equation}\label{equequinlapform}
\begin{aligned}
    &z\tilde{G}_1+az^{1-\alpha_1}\tilde{G}_1+z^{1-\alpha_1}A_1\tilde{G}_1=az^{1-\alpha_2}\tilde{G}_2+G_{1,0},\\
    &z\tilde{G}_2+az^{1-\alpha_2}\tilde{G}_2+z^{1-\alpha_2}A_2\tilde{G}_2=az^{1-\alpha_1}\tilde{G}_1+G_{2,0}.
\end{aligned}
\end{equation}
Denote
\begin{equation}\label{equdefHz}
 H(z,A,\alpha,\beta)=z^\beta\left(z^{\alpha}+a+A\right)^{-1},
\end{equation}
where $A$ is an operator. Then from \eqref{equequinlapform} and \eqref{equdefHz} we have
\begin{equation}\label{equsolrepinlap0}
    \begin{aligned}
        \tilde{G}_1=&H(z,A_1,\alpha_1,\alpha_1-1)G_{1,0}+aH(z,A_1,\alpha_1,\alpha_1-\alpha_2)\tilde{G}_{2},\\
        \tilde{G}_2=&H(z,A_2,\alpha_2,\alpha_2-1)G_{2,0}+aH(z,A_2,\alpha_2,\alpha_2-\alpha_1)\tilde{G}_{1}.\\
    \end{aligned}
\end{equation}
Thus
\begin{equation}\label{equsolrepinlap1}
\begin{aligned}
\tilde{G}_1=&H(z,A_1,\alpha_1,\alpha_1-1)G_{1,0}+aH(z,A_1,\alpha_1,\alpha_1-\alpha_2)(H(z,A_2,\alpha_2,\alpha_2-1)G_{2,0}\\
&+aH(z,A_2,\alpha_2,\alpha_2-\alpha_1)\tilde{G}_{1}),\\
\tilde{G}_2=&H(z,A_2,\alpha_2,\alpha_2-1)G_{2,0}
+aH(z,A_2,\alpha_2,\alpha_2-\alpha_1)(H(z,A_1,\alpha_1,\alpha_1-1)G_{1,0}\\
&+aH(z,A_1,\alpha_1,\alpha_1-\alpha_2)\tilde{G}_{2}).
\end{aligned}
\end{equation}
\begin{lemma}\label{lemtheestimateofHblabla}
    Let $A$ be the fractional Laplacian $(-\Delta)^s$ with homogeneous Dirichlet boundary condition. When $z\in\Sigma_{\theta,\kappa}$, $ \pi/2<\theta<\pi $ and $\kappa$ is large enough, we have the estimates
    \begin{equation*}
    \begin{aligned}
    &\| H(z,A,\alpha,\beta)\|_{L^2(\Omega)\rightarrow L^2(\Omega)}\leq C|z|^{\beta-\alpha},\  \| AH(z,A,\alpha,\beta)\|_{L^2(\Omega)\rightarrow L^2(\Omega)}\leq C|z|^{\beta},
    \end{aligned}
    \end{equation*}
    where $H(z,A,\alpha,\beta)$ is defined in \eqref{equdefHz}.
\end{lemma}
\begin{proof}
Let $u=H(z,A,\alpha,\beta)v$. By simple calculations, we obtain
\begin{equation*}
u=(z^{\alpha}+A)^{-1}(-au+z^{\beta}v).
\end{equation*}
Taking $L^2$ norm on both sides and using the resolvent estimates provided in \cite{Acosta20174}, we have
\begin{equation*}
\|u\|_{L^2(\Omega)}\leq C|z|^{-\alpha}(|a|\|u\|_{L^2(\Omega)}+|z|^{\beta}\|v\|_{L^2(\Omega)}),
\end{equation*}
which leads to the first desired estimate by taking $\kappa$ large enough and $|z|>\kappa$. Since $AH(z,A,\alpha,\beta)=z^{\beta}(I-(z^\alpha+a) H(z,A,\alpha,0))$, it can be easily got the second estimate.
\end{proof}

Then we provide the resolvent estimate in $\hat{H}^{1/2+s-\epsilon}(\Omega)$.
\begin{lemma}\label{lemtheestimateofHblablainH}
	Let $A$ be the fractional Laplacian $(-\Delta)^s$ with homogeneous Dirichlet boundary condition and $s<1/2$. When $z\in\Sigma_{\theta,\kappa}$, $ \pi/2<\theta<\pi $ and $\kappa$ is large enough, we have the estimates
	\begin{equation*}
	\begin{aligned}
	&\| H(z,A,\alpha,\beta)\|_{\hat{H}^{\sigma}(\Omega)\rightarrow \hat{H}^{\sigma}(\Omega)}\leq C|z|^{\beta-\alpha},\  \| AH(z,A,\alpha,\beta)\|_{\hat{H}^{\sigma}(\Omega)\rightarrow \hat{H}^{\sigma}(\Omega)}\leq C|z|^{\beta},
	\end{aligned}
	\end{equation*}
	where $H(z,A,\alpha,\beta)$ is defined in \eqref{equdefHz} and $\sigma\in[0,1/2+s)$. Furthermore, there exists
	\begin{equation*}
	\|AH(z,A,\alpha,\beta)\|_{_{\hat{H}^{\tilde{\sigma}+2\mu s}(\Omega)\rightarrow \hat{H}^{\tilde{\sigma}}(\Omega)}}\leq C|z|^{(\beta-\mu\alpha)}, 
	\end{equation*}
	where $\tilde{\sigma}\in[0,1/2)$ and $\mu \in [0,1]$.
\end{lemma}
\begin{proof}
	Assume $u=H(z,A,\alpha,\beta)v$ and $v=0$ in $\Omega^c$.  Using Theorem \ref{thmregdiri} and Lemma \ref{lemtheestimateofHblabla}, we have
	\begin{equation} \label{eqnorm}
    \begin{aligned}
	\|u\|_{\hat{H}^{2s}(\Omega)}\leq\|Au\|_{L^2(\Omega)} & =\|AH(z,A,\alpha,\beta)v\|_{L^2(\Omega)} \\
 & \leq C|z|^{\beta-\alpha}\|Av\|_{L^2(\Omega)}\leq C|z|^{\beta-\alpha}\|v\|_{\hat{H}^{2s}(\Omega)},
    \end{aligned}
	\end{equation}
which leads to
	\begin{equation*}
	\|H(z,A,\alpha,\beta)\|_{\hat{H}^{2s}(\Omega)\rightarrow \hat{H}^{2s}(\Omega)}\leq C|z|^{\beta-\alpha}.
	\end{equation*}
By induction, one can get 
     \begin{equation*}
	\|H(z,A,\alpha,\beta)\|_{\hat{H}^{2 m s}(\Omega)\rightarrow \hat{H}^{2m s}(\Omega)}\leq C|z|^{\beta-\alpha}~~ {\rm for} ~~(2m-1)s<1/2,
	\end{equation*}
where $m$ is a positive integer. By Lemma \ref{lemtheestimateofHblabla} and the interpolation property \cite{Adams2003}, there exists
	\begin{equation*}
	\|H(z,A,\alpha,\beta)\|_{\hat{H}^{\sigma}(\Omega)\rightarrow \hat{H}^{\sigma}(\Omega)}\leq C|z|^{\beta-\alpha},\quad \sigma\in[0,2 m s].
	\end{equation*}
Taking $\sigma=1/2-s-\epsilon$ in the above equation and using (\ref{eqnorm}), there is
	\begin{equation*}
	\|H(z,A,\alpha,\beta)\|_{\hat{H}^{1/2+s-\epsilon}(\Omega)\rightarrow \hat{H}^{1/2+s-\epsilon}(\Omega)}\leq C|z|^{\beta-\alpha}.
	\end{equation*}
	Similarly by interpolation property, one can obtain
	\begin{equation*}
	\|H(z,A,\alpha,\beta)\|_{\hat{H}^{\sigma}(\Omega)\rightarrow \hat{H}^{\sigma}(\Omega)}\leq C|z|^{\beta-\alpha},\quad \sigma\in [0,1/2+s).
	\end{equation*}
Noting that $AH(z,A,\alpha,\beta)=z^{\beta}(I-(z^\alpha+a) H(z,A,\alpha,0))$, the second estimate can be got. On the other hand, let $u=AH(z,A,\alpha,\beta)v$ and $v=0$ in $\Omega^c$. For $\tilde{\sigma}\in[0,1/2)$, we have
	\begin{equation*}
		\|u\|_{\hat{H}^{\tilde{\sigma}}(\Omega)}=\|AH(z,A,\alpha,\beta)v\|_{\hat{H}^{\tilde{\sigma}}(\Omega)}\leq C|z|^{\beta-\alpha}\|Av\|_{\hat{H}^{\tilde{\sigma}}(\Omega)}\leq C|z|^{\beta-\alpha}\|v\|_{\hat{H}^{{\tilde{\sigma}}+2s}(\Omega)},
	\end{equation*}
	 which leads to $\|AH(z,A,\alpha,\beta)\|_{_{\hat{H}^{\tilde{\sigma}+2s}(\Omega)\rightarrow \hat{H}^{\tilde{\sigma}}(\Omega)}}\leq C|z|^{\beta-\alpha}$. Using the property of interpolation, we obtain
	\begin{equation*}
		\|AH(z,A,\alpha,\beta)\|_{_{\hat{H}^{\tilde{\sigma}+2\mu s}(\Omega)\rightarrow \hat{H}^{\tilde{\sigma}}(\Omega)}}\leq C|z|^{(\beta-\mu\alpha)}, \quad \mu \in [0,1].
	\end{equation*}
\end{proof}

\begin{lemma}\label{lemestZ}
Let $\kappa$ satisfy the conditions given in Lemma \ref{lemtheestimateofHblabla} and $\Omega\subset \mathbb{R}^n$. Then we have the estimate
	\begin{equation*}
	\int_{\Gamma_{\theta,\kappa}}|e^{zt}||z|^\alpha |dz|\leq Ct^{-\alpha-1}.\\
	\end{equation*}
\end{lemma}
\begin{proof}
	By simple calculation and taking $r=|z|$, we have
\begin{equation*}
\begin{aligned}
&\int_{\Gamma_{\theta,\kappa}}|e^{zt}||z|^\alpha |dz|
=\int_{\kappa}^{\infty}e^{r\cos(\theta)t}r^{\alpha}dr+\kappa^{1+\alpha}\int_{-\theta}^{\theta}e^{\kappa\cos(\eta) t}d\eta
\leq Ct^{-\alpha-1}+C\kappa^{1+\alpha}.
\end{aligned}
\end{equation*}
When $\alpha\geq-1$, using the fact $T/t>1$, we can get the desired estimate. And when $\alpha<-1$, the desired estimate can be got by taking $\kappa>1/t$.
\end{proof}

Next, we provide the following Gr\"{o}nwall inequality which is similar to the one provided in \cite{Laesson1992}.
\begin{lemma}\label{lemgrondwall}
    Let the function $\phi(t)\geq 0$ be continuous for $0< t\leq T$. If
    \begin{equation*}
        \phi(t)\leq \sum_{k=1}^Na_kt^{-1+\alpha_k}+b\int_0^t(t-s)^{-1+\beta}\phi(s)ds,\ 0<t\leq T,
    \end{equation*}
    for some constants $\{a_k\}_{k=0}^N$, $\{\alpha_k\}_{k=0}^N$, $b\geq0$,  $\beta>0$, then there is a constant $C=C(b,T,\alpha,\beta)$ such that
    \begin{equation*}
        \phi(t)\leq C\sum_{k=1}^Na_kt^{-1+\alpha_k} ~~{\rm for } ~~  0<t\leq T.
    \end{equation*}
\end{lemma}
\begin{proof}
The proof is similar to the one provided in \cite{Laesson1992}.
\end{proof}

Then we present the priori estimates for the solutions $G_1$ and $G_2$ of \eqref{equrqtosol} with nonsmooth initial value.
\begin{theorem}\label{thmhomoregularitynonsmooth}
   Let $\gamma_1<\min(s_1,1/2-\epsilon)$ and $\gamma_2<\min(s_2,1/2-\epsilon)$. If $G_{1,0},\ G_{2,0}\in L^2(\Omega)$,  then we have
    \begin{equation*}
    \begin{aligned}
    \|G_1(t)\|_{L^2(\Omega)}\leq& C\|G_{1,0}\|_{L^2(\Omega)}+C\|G_{2,0}\|_{L^2(\Omega)},\\
    \|G_2(t)\|_{L^2(\Omega)}\leq& C\|G_{1,0}\|_{L^2(\Omega)}+C\|G_{2,0}\|_{L^2(\Omega)};
    \end{aligned}
    \end{equation*}
    and
    \begin{equation*}
    \begin{aligned}
    \|G_1(t)\|_{\hat{H}^{s_1+\gamma_1}(\Omega)}\leq Ct^{-\alpha_1}\|G_{1,0}\|_{L^2(\Omega)}+Ct^{\min(0,\alpha_2-\alpha_1)}\|G_{2,0}\|_{L^2(\Omega)},\\
    \|G_2(t)\|_{\hat{H}^{s_2+\gamma_2}(\Omega)}\leq Ct^{\min(0,\alpha_1-\alpha_2)}\|G_{1,0}\|_{L^2(\Omega)}+Ct^{-\alpha_2}\|G_{2,0}\|_{L^2(\Omega)}.
    \end{aligned}
    \end{equation*}
\end{theorem}
\begin{proof}
    By Eq. \eqref{equsolrepinlap1}, Lemmas \ref{lemtheestimateofHblabla}, \ref{lemestZ} and taking the inverse Laplace transform for \eqref{equsolrepinlap1}, we obtain
    \begin{equation*}
    \begin{aligned}
    \|G_1(t)\|_{L^2(\Omega)}\leq& C\|G_{1,0}\|_{L^2(\Omega)}+Ct^{\alpha_2}\|G_{2,0}\|_{L^2(\Omega)}+\int_0^t(t-s)^{\alpha_1+\alpha_2-1}\|G_1\|_{L^2(\Omega)}ds,\\
    \|G_2(t)\|_{L^2(\Omega)}\leq& Ct^{\alpha_1}\|G_{1,0}\|_{L^2(\Omega)}+C\|G_{2,0}\|_{L^2(\Omega)}+\int_0^t(t-s)^{\alpha_1+\alpha_2-1}\|G_2\|_{L^2(\Omega)}ds.
    \end{aligned}
    \end{equation*}
    According to Lemma \ref{lemgrondwall} and the fact $T/t>1$, one can get the desired $L^2$ estimates. Similarly, acting $A_i$ on both sides of Eq. \eqref{equsolrepinlap1} respectively and using Lemmas \ref{lemtheestimateofHblabla}, \ref{lemestZ}, one can obtain
    \begin{equation*}
    \begin{aligned}
   & \|G_1(t)\|_{\hat{H}^{s_1+\gamma_1}(\Omega)}
    \\
   & \leq Ct^{-\alpha_1}\|G_{1,0}\|_{L^2(\Omega)}+Ct^{\alpha_2-\alpha_1}\|G_{2,0}\|_{L^2(\Omega)}+\int_0^t(t-s)^{\alpha_2-1}\|G_1\|_{L^2(\Omega)}ds,\\
   & \|G_2(t)\|_{\hat{H}^{s_2+\gamma_2}(\Omega)}
    \\
   & \leq Ct^{\alpha_1-\alpha_2}\|G_{1,0}\|_{L^2(\Omega)}+Ct^{-\alpha_2}\|G_{2,0}\|_{L^2(\Omega)}+\int_0^t(t-s)^{\alpha_1-1}\|G_2\|_{L^2(\Omega)}ds.
    \end{aligned}
    \end{equation*}
     In view of the $L^2$ estimates of $G_1$ and $G_2$ proved above and $T/t>1$, the desired estimates can be got.
\end{proof}

Lastly, we provide a detailed discussion on the regularity of the solutions when $s_1<1/2$.

\begin{theorem}\label{thmhomoregularitysmooth1}
	  Assume $s_1\leq s_2<1/2$. If $G_{1,0},\ G_{2,0}\in \hat{H}^\sigma(\Omega)$, $\sigma<1/2$,  then we have
	 \begin{equation*}
	 \begin{aligned}
	 &\|G_1(t)\|_{\hat{H}^{s_1+\gamma_1}(\Omega)}\leq Ct^{-\alpha_1}\|G_{1,0}\|_{\hat{H}^\sigma(\Omega)}+Ct^{\min(0,\alpha_2-\alpha_1)}\|G_{2,0}\|_{\hat{H}^\sigma(\Omega)},\\
	 &\|G_2(t)\|_{\hat{H}^{s_2+\gamma_2}(\Omega)}\leq Ct^{\min(0,\alpha_1-\alpha_2)}\|G_{1,0}\|_{\hat{H}^\sigma(\Omega)}+Ct^{-\alpha_2}\|G_{2,0}\|_{\hat{H}^\sigma(\Omega)},
	 \end{aligned}
	 \end{equation*}
where $\gamma_i=\min(1/2-\epsilon,s_i+\sigma)$ $(i=1,2)$.
\end{theorem}
\begin{remark}
The proof of Theorem \ref{thmhomoregularitysmooth1} is similar to the one of Theorem \ref{thmhomoregularitynonsmooth}.
\end{remark}
\begin{theorem}\label{thmhomoregularitysmooth2}
	 Assume $s_1\leq s_2<1/2$, $G_{i,0}\in \hat{H}^{\sigma_i}(\Omega)$, and $\sigma_i< 1/2$ $(i=1,2)$. Denote $\mu_1=\max(\frac{\alpha_1-\alpha_2}{\alpha_1}+\epsilon,0)$, $\mu_2=\max(\frac{\alpha_2-\alpha_1}{\alpha_2}+\epsilon,0)$, and $\bar{\gamma}_i=\min(1/2-\epsilon,s_i+\sigma_i)$ $(i=1,2)$.
\begin{itemize}
\item If $\sigma_1+2\mu_1s_1<s_2+\bar{\gamma}_2$ and $\sigma_2+2\mu_2s_2<s_1+\bar{\gamma}_1$, then we have
\begin{equation*}
	 \begin{aligned}
	 &\|G_1(t)\|_{\hat{H}^{s_1+\gamma_1}(\Omega)}\leq Ct^{-\alpha_1}\|G_{1,0}\|_{\hat{H}^{\sigma_1}(\Omega)}+C\|G_{2,0}\|_{\hat{H}^{\sigma_2}(\Omega)},\\
	 &\|G_2(t)\|_{\hat{H}^{s_2+\gamma_2}(\Omega)}\leq C\|G_{1,0}\|_{\hat{H}^{\sigma_1}(\Omega)}+Ct^{-\alpha_2}\|G_{2,0}\|_{\hat{H}^{\sigma_2}(\Omega)};
	 \end{aligned}
\end{equation*}
\item Assume $\sigma_1>\sigma_2$. If $\sigma_1+2\mu_1s_1>s_2+\bar{\gamma}_2$ or $\sigma_2+2\mu_2s_2>s_1+\bar{\gamma}_1$, then we get
\begin{equation*}
	 \begin{aligned}
	 &\|G_1(t)\|_{\hat{H}^{s_1+\gamma_1}(\Omega)}\leq Ct^{-\alpha_1}\|G_{1,0}\|_{\hat{H}^{\sigma_1}(\Omega)}+C\|G_{2,0}\|_{\hat{H}^{\sigma_2}(\Omega)},\\
	 &\|G_2(t)\|_{\hat{H}^{s_2+\bar{\gamma}_2}(\Omega)}\leq Ct^{\min(0,\alpha_1-\alpha_2)}\|G_{1,0}\|_{\hat{H}^{\sigma_1}(\Omega)}+Ct^{-\alpha_2}\|G_{2,0}\|_{\hat{H}^{\sigma_2}(\Omega)};
	 \end{aligned}
\end{equation*}
\item Assume $\sigma_1<\sigma_2$. If $\sigma_1+2\mu_1s_1>s_2+\bar{\gamma}_2$ or $\sigma_2+2\mu_2s_2>s_1+\bar{\gamma}_1$, then we obtain
\begin{equation*}
	 \begin{aligned}
	 &\|G_1(t)\|_{\hat{H}^{s_1+\bar{\gamma}_1}(\Omega)}\leq Ct^{-\alpha_1}\|G_{1,0}\|_{\hat{H}^{\sigma_1}(\Omega)}+Ct^{\min(0,\alpha_2-\alpha_1)}\|G_{2,0}\|_{\hat{H}^{\sigma_2}(\Omega)},\\
	 &\|G_2(t)\|_{\hat{H}^{s_2+\gamma_2}(\Omega)}\leq C\|G_{1,0}\|_{\hat{H}^{\sigma_1}(\Omega)}+Ct^{-\alpha_2}\|G_{2,0}\|_{\hat{H}^{\sigma_2}(\Omega)}.
	 \end{aligned}
\end{equation*}
\end{itemize}
 Here $\gamma_1=\min(1/2-\epsilon,s_1+\sigma_1,s_1+s_2+\bar{\gamma}_2-2\mu_1s_1)$ and $\gamma_2=\min(1/2-\epsilon,s_2+\sigma_2,s_2+s_1+\bar{\gamma}_1-2\mu_2s_2)$.
 \end{theorem}

\begin{proof}
 When $\sigma_1+2\mu_1s_1\leq s_2+\bar{\gamma}_2$ and $\sigma_2+2\mu_2s_2\leq s_1+\bar{\gamma}_1$, acting $A_i$ on both sides of \eqref{equsolrepinlap0} respectively, according to Theorem \ref{thmregdiri}, embedding theorem, Lemmas \ref{lemtheestimateofHblablainH}, \ref{lemestZ} and taking the inverse Laplace transform of \eqref{equsolrepinlap0}, there are
\begin{equation*}
\begin{aligned}
& \|G_1(t)\|_{\hat{H}^{s_1+\gamma_1}(\Omega)}
\\
& \leq Ct^{-\alpha_1}\|G_{1,0}\|_{\hat{H}^{\sigma_1}(\Omega)}+\int_0^t(t-s)^{(\mu_1-1)\alpha_1+\alpha_2}\|G_2\|_{\hat{H}^{\sigma_1+2\mu_1s_1}(\Omega)}ds,\\
& \|G_2(t)\|_{\hat{H}^{s_2+\gamma_2}(\Omega)}
\\ &
\leq Ct^{-\alpha_2}\|G_{2,0}\|_{\hat{H}^{\sigma_2}(\Omega)}+\int_0^t(t-s)^{(\mu_2-1)\alpha_2+\alpha_1}\|G_1\|_{\hat{H}^{\sigma_2+2\mu_2s_2}(\Omega)}ds,
\end{aligned}
\end{equation*}
where $\gamma_i=\min(1/2-\epsilon,s_i+\sigma_i)$ $(i=1,2)$. Thus by Lemma \ref{lemgrondwall}, embedding theorem \cite{Adams2003} and the fact $T/t>1$, the desired estimates can be got.

If $\sigma_1+2\mu_1s_1> s_2+\bar{\gamma}_2$ or $\sigma_2+2\mu_2s_2> s_1+\bar{\gamma}_1$, consider $\sigma_1>\sigma_2$ first. According to Theorem \ref{thmhomoregularitysmooth1}, there exists
\begin{equation*}
\|G_2(t)\|_{\hat{H}^{s_2+\bar{\gamma}_2}(\Omega)}\leq Ct^{\alpha_1-\alpha_2}\|G_{1,0}\|_{\hat{H}^{\sigma_1}(\Omega)}+Ct^{-\alpha_2}\|G_{2,0}\|_{\hat{H}^{\sigma_2}(\Omega)}.
\end{equation*}
Thus
\begin{equation*}
\|G_1(t)\|_{\hat{H}^{s_1+\gamma_1}(\Omega)}\leq Ct^{-\alpha_1}\|G_{1,0}\|_{\hat{H}^{\sigma_1}(\Omega)}+\int_0^t(t-s)^{(\mu_1-1)\alpha_1+\alpha_2}\|G_2\|_{\hat{H}^{s_2+\bar{\gamma}_2}(\Omega)}ds,
\end{equation*}
where $\gamma_1=\min(1/2-\epsilon,s_1+s_2+\bar{\gamma}_2-2\mu_1s_1,s_1+\sigma_1)$. Similarly for $\sigma_1<\sigma_2$, one has
\begin{equation*}
\|G_1(t)\|_{\hat{H}^{s_1+\bar{\gamma}_1}(\Omega)}\leq Ct^{\alpha_2-\alpha_1}\|G_{2,0}\|_{\hat{H}^{\sigma_2}(\Omega)}+Ct^{-\alpha_1}\|G_{1,0}\|_{\hat{H}^{\sigma_1}(\Omega)}.
\end{equation*}
So
\begin{equation*}
\|G_2(t)\|_{\hat{H}^{s_2+\gamma_2}(\Omega)}\leq Ct^{-\alpha_2}\|G_{2,0}\|_{\hat{H}^{\sigma_2}(\Omega)}+\int_0^t(t-s)^{(\mu_2-1)\alpha_2+\alpha_1}\|G_1\|_{\hat{H}^{s_1+\bar{\gamma}_1}(\Omega)}ds,
\end{equation*}
where $\gamma_2=\min(1/2-\epsilon,s_2+s_1+\bar{\gamma}_1-2\mu_2s_2,s_2+\sigma_2)$. Thus by embedding theorem \cite{Adams2003} and the fact $T/t>1$, the desired estimates are obtained.
\end{proof}

\begin{theorem}\label{thmhomoregularitysmooth3}
	 Assume $s_1<1/2\leq s_2$. If $G_{1,0}\in \hat{H}^{\sigma_i}(\Omega)$, $\sigma_1< 1/2-s_1$ and $\sigma_2=0$, then
\begin{equation*}
	 \begin{aligned}
	 &\|G_1(t)\|_{\hat{H}^{s_1+\gamma_1}(\Omega)}\leq Ct^{-\alpha_1}\|G_{1,0}\|_{\hat{H}^{\sigma_1}(\Omega)}+C\|G_{2,0}\|_{\hat{H}^{\sigma_2}(\Omega)},\\
    &\|G_2(t)\|_{\hat{H}^{s_2+\gamma_2}(\Omega)}\leq Ct^{\min(0,\alpha_1-\alpha_2)}\|G_{1,0}\|_{\hat{H}^{\sigma_1}(\Omega)}+Ct^{-\alpha_2}\|G_{2,0}\|_{\hat{H}^{\sigma_2}(\Omega)},
	 \end{aligned}
	 \end{equation*}
 where $\gamma_i=\min(1/2-\epsilon,s_i+\sigma_i)$ $(i=1,2)$.
 \end{theorem}
\begin{remark}
Combining the proofs of Theorems \ref{thmhomoregularitynonsmooth} and \ref{thmhomoregularitysmooth2}, Theorem \ref{thmhomoregularitysmooth3} can be obtained.
\end{remark}

\begin{remark}
For Eq. \eqref{tramod}, one can obtain that $\|G\|_{\hat{H}^{s+1/2-\epsilon}(\Omega)}\leq Ct^{-\alpha}\|G_0\|_{L^2(\Omega)}$ for $s\in[1/2,1)$ and $\|G\|_{\hat{H}^{s+\gamma}(\Omega)}\leq Ct^{-\alpha}\|G_0\|_{\hat{H}^{\sigma}(\Omega)}$  for $s\in(0,1/2)$, where $\gamma=\min(1/2-\epsilon,s+\sigma)$, $\sigma>0$.
\end{remark}

\section{Space discretization and error analysis} \label{Sec2}
In this section, we discretize the fractional Laplacian by the finite element method and provide the error estimates for the space semidiscrete scheme of system \eqref{equrqtosol}. Let $\mathcal{T}_h$ be a shape regular quasi-uniform partitions of the domain $\Omega$, where $h$ is the maximum diameter. Denote $ X_h $ as the piecewise linear finite element space
\begin{equation*}
X_{h}=\{v_h\in C(\bar{\Omega}): v_h|_\mathbf{T}\in \mathcal{P}^1,\  \forall \mathbf{T}\in\mathcal{T}_h,\ v_h|_{\partial \Omega}=0\},
\end{equation*}
where $\mathcal{P}^1$ denotes the set of piecewise polynomials of degree $1$ over $\mathcal{T}_h$. Then we define the $ L^2 $-orthogonal projection $ P_h:\ L^2(\Omega)\rightarrow X_h $ by
\begin{equation*}
\begin{aligned}
&(P_hu,v_h)=(u,v_h) \ ~\forall v_h\in X_h,
\end{aligned}
\end{equation*}
which has the the following approximation property.
\begin{lemma}[\cite{Bazhlekova2015}]\label{lemprojection}
    The projection $ P_h $ satisfies
    \begin{equation*}
    \begin{aligned}
    &\|P_hu-u\|_{L^2(\Omega)}+h\|\nabla(P_hu-u)\|_{L^2(\Omega)}\leq Ch^q\|u\|_{H^{q}(\Omega)}\ ~~for\ u\in H^q(\Omega),\,\ q=1,2.\\
    \end{aligned}
    \end{equation*}
\end{lemma}

Denote $(\cdot,\cdot)$ as the $L_2$ inner product. The semidiscrete Galerkin scheme for system \eqref{equrqtosol} reads: Find $ G_{1,h}\in X_h$ and $G_{2,h}\in X_h$ such that
\begin{equation}\label{equequsysspatialsemischeme}
\left\{\begin{aligned}
&\left (\frac{\partial G_{1,h}}{\partial t},v_{1,h}\right )+a~_0D^{1-\alpha_1}_t(G_{1,h},v_{1,h})+~_0D^{1-\alpha_1}_t\langle G_{1,h}, v_{1,h}\rangle_{s_1}
\\
&=a~_0D^{1-\alpha_2}_t(G_{2,h},v_{1,h}),\\
&\left (\frac{\partial G_{2,h}}{\partial t},v_{2,h}\right )+a~_0D^{1-\alpha_2}_t(G_{2,h},v_{2,h})+~_0D^{1-\alpha_2}_t\langle G_{2,h}, v_{2,h}\rangle_{s_2}
\\
&=a~_0D^{1-\alpha_1}_t(G_{1,h},v_{2,h}),\\
\end{aligned}\right.
\end{equation}
where $v_{1,h},\ v_{2,h}\in X_h$. As for $G_{1,h}(0)$ and $ G_{2,h}(0) $, we take $G_{1,h}(0)=P_hG_{1,0}$, $G_{2,h}(0)=P_hG_{2,0}$.

Define the discrete operators $A_{i,h}$: $X_h\rightarrow X_h$ as
\begin{equation*}
(A_{i,h} u_h,v_h)=\langle u_h,v_h\rangle_{s_i} \quad\forall u_h,\ v_h\in X_h,\ i=1,2.
\end{equation*}
Then  \eqref{equequsysspatialsemischeme}  can be rewritten as
\begin{equation}\label{equequsysspatialsemiAh}
\begin{aligned}
&\frac{\partial G_{1,h}}{\partial t}+a~_0D^{1-\alpha_1}_tG_{1,h}+~_0D^{1-\alpha_1}_tA_{1,h} G_{1,h}=a~_0D^{1-\alpha_2}_tG_{2,h},\\
&\frac{\partial G_{2,h}}{\partial t}+a~_0D^{1-\alpha_2}_tG_{2,h}+~_0D^{1-\alpha_2}_tA_{2,h} G_{2,h}=a~_0D^{1-\alpha_1}_tG_{1,h}.\\
\end{aligned}
\end{equation}
Taking the Laplace transforms of \eqref{equequsysspatialsemiAh}, we get
\begin{equation}\label{equequsysspatialsemilapform}
\begin{aligned}
&z\tilde{G}_{1,h}+az^{1-\alpha_1}\tilde{G}_{1,h}+z^{1-\alpha_1}A_{1,h}\tilde{G}_{1,h}=az^{1-\alpha_2}\tilde{G}_{2,h}+G_{1,h}(0),\\
&z\tilde{G}_{2,h}+az^{1-\alpha_2}\tilde{G}_{2,h}+z^{1-\alpha_2}A_{2,h}\tilde{G}_{2,h}=az^{1-\alpha_1}\tilde{G}_{1,h}+G_{2,h}(0).
\end{aligned}
\end{equation}

Next we introduce two lemmas, which will be used in the error estimate between system \eqref{equrqtosol} and  space semidiscrete scheme \eqref{equequsysspatialsemischeme}.

\begin{lemma}\label{lemestofapb}
     For any $\phi\in \hat{H}^s(\Omega)$, $z\in \Sigma_{\theta,\kappa}$ with $\theta\in(\pi/2,\pi)$ and $\kappa$ being taken to be large enough to ensure $g(z)=z^{\alpha}+a\in \Sigma_{\theta}$, there exists
    \begin{equation*}
    |g(z)|\|\phi\|^2_{L^2(\Omega)}+\|\phi\|^2_{\hat{H}^{s}(\Omega)}\leq C\left |g(z)\|\phi\|^2_{L^2(\Omega)}+\|\phi\|^2_{\hat{H}^{s}(\Omega)}\right |.
    \end{equation*}
\end{lemma}
\begin{lemma}\label{lemeroper}
    Let $v\in L^2(\Omega)$, $A=(-\Delta)^s$ with  homogeneous Dirichlet boundary condition, and $z\in\Gamma_{\theta,\kappa}$. Denote $w=H(z,A,\alpha,0)v$ and $w_h=H(z,A_h,\alpha,0)P_hv$.  Then one has
    \begin{equation*}
    \|w-w_h\|_{L^2(\Omega)}+h^\gamma\|w-w_h\|_{\hat{H}^s(\Omega)}\leq Ch^{2\gamma}\|v\|_{L^2(\Omega)},
    \end{equation*}
    where
    \begin{equation*}
        (A_{h} u_h,v_h)=\langle u_h,v_h\rangle_{s} \quad\forall u_h,\ v_h\in X_h,
    \end{equation*}
     and $\gamma\leq \min(s,1/2-\epsilon)$ with $\epsilon>0$ being arbitrarily small.
\end{lemma}

\begin{remark}
The proofs of Lemmas \ref{lemestofapb} and \ref{lemeroper} are similar to the ones in \cite{Acosta20174,Bazhlekova2015}.
\end{remark}

When $v\in \hat{H}^{\sigma}(\Omega)$, we modify the estimate in Lemma \ref{lemeroper} as
\begin{lemma}\label{lemeroper2}
    Let $A=(-\Delta)^s$ with  homogeneous Dirichlet boundary condition, $s<1/2$,  and $z\in\Gamma_{\theta,\kappa}$. Assume $v\in \hat{H}^{\sigma}(\Omega)$ with $\sigma<1/2-s$. Denote $w=H(z,A,\alpha,0)v$ and $w_h=H(z,A_h,\alpha,0)P_hv$. Then there exists
    \begin{equation*}
    \|w-w_h\|_{L^2(\Omega)}+h^s\|w-w_h\|_{\hat{H}^s(\Omega)}\leq Ch^{2s+\sigma}\|v\|_{\hat{H}^{\sigma}(\Omega)},
    \end{equation*}
    where
    \begin{equation*}
        (A_{h} u_h,v_h)=\langle u_h,v_h\rangle_{s}  \quad\forall u_h,\ v_h\in X_h.
    \end{equation*}
\end{lemma}
\begin{proof}
Let $\epsilon>0$ be arbitrarily small. Here we first consider $\sigma=1/2-s-\epsilon$. Using the notation of $g(z)$ in Lemma \ref{lemestofapb} and the definitions of $w$ and $w_h$, there exist
    \begin{equation*}
    \begin{aligned}
    &g(z)(w,\chi)+\langle w,\chi\rangle_s=(v,\chi)\quad \forall \chi\in \hat{H}^s(\Omega),\\
    &g(z)(w_h,\chi)+\langle w_h,\chi\rangle_s=(v,\chi)\quad \forall \chi\in X_h.\\
    \end{aligned}
    \end{equation*}
    Thus
    \begin{equation*}
    g(z)(e,\chi)+\langle e,\chi\rangle_s=0\quad \forall \chi\in X_h,
    \end{equation*}
    where $e=w-w_h$.
    By Lemma \ref{lemestofapb}, one has
    \begin{equation*}
    \begin{aligned}
    |g(z)|\|e\|^2_{L^2(\Omega)}+\|e\|^2_{\hat{H}^s(\Omega)}\leq& C\left |g(z)\|e\|^2_{L^2(\Omega)}+\|e\|^2_{\hat{H}^s(\Omega)}\right |\\
    =&C\left |g(z)(e,w-\chi)+\langle e,(w-\chi)\rangle_s\right |.
    \end{aligned}
    \end{equation*}
    Taking $\chi=\pi_h w$ as the Lagrange interpolation of $w$ and using the Cauchy-Schwarz inequality, we obtain
    \begin{equation*}
    \begin{aligned}
   & |g(z)|\|e\|^2_{L^2(\Omega)}+\|e\|^2_{\hat{H}^s(\Omega)}
      \\	
   & \leq C h^{1/2-\epsilon}|g(z)|\|e\|_{L^2(\Omega)}\|w\|_{\hat{H}^{1/2-\epsilon}(\Omega)}+\|e\|_{\hat{H}^{s}(\Omega)}h^{1/2-\epsilon}\|w\|_{\hat{H}^{s+1/2-\epsilon}(\Omega)}.
     \end{aligned}
    \end{equation*}
    According to Lemma \ref{lemestofapb}, there is 
    \begin{equation*}
    \begin{aligned}
    |g(z)|\|w\|^2_{L^2(\Omega)}+\|w\|^2_{\hat{H}^s(\Omega)}\leq& C|((g(z)+A)w,w)|
    \\	
    \leq& C\|v\|_{L^2(\Omega)}\|w\|_{L^2(\Omega)}.
    \end{aligned}
    \end{equation*}
    Thus
    \begin{equation*}
    \begin{aligned}
    &\|w\|_{L^2(\Omega)}\leq C|g(z)|^{-1}\|v\|_{L^2(\Omega)},\quad\| w\|^2_{\hat{H}^s(\Omega)}\leq C|g(z)|^{-1}\|v\|^2_{L^2(\Omega)}.
    \end{aligned}	
    \end{equation*}
    Therefore, $\|(g(z)+A)^{-1}\|_{L^2(\Omega)\rightarrow \hat{H}^s(\Omega)}\leq C|g(z)|^{-1/2}$. Similar to Lemma $\ref{lemtheestimateofHblablainH}$, there exist
    \begin{equation*}
    	\begin{aligned}
    	&\|w\|_{\hat{H}^{2s+\sigma}(\Omega)}\leq\|Aw\|_{\hat{H}^{\sigma}(\Omega)}\leq \|A(g(z)+A)^{-1}v\|_{\hat{H}^{\sigma}(\Omega)}\leq \|v\|_{\hat{H}^{\sigma}(\Omega)},\\
    	&\|w\|_{\hat{H}^{2s+\sigma}(\Omega)}\leq\|Aw\|_{\hat{H}^{\sigma}(\Omega)}\leq \|A(g(z)+A)^{-1}v\|_{\hat{H}^{\sigma}(\Omega)}\leq |g(z)|^{-1}\|v\|_{\hat{H}^{2s+\sigma}(\Omega)}.
    	\end{aligned}
    \end{equation*}
    Using the interpolation property, we get
    \begin{equation*}
    	\|w\|_{\hat{H}^{2s+\sigma}(\Omega)}\leq |g(z)|^{-1/2}\|v\|_{\hat{H}^{s+\sigma}(\Omega)}.
    \end{equation*}
   Further using the interpolation property leads to
    \begin{equation*}
    \|w\|_{\hat{H}^{1/2-\epsilon}(\Omega)}\leq C |g(z)|^{-1/2}\|v\|_{\hat{H}^{1/2-s-\epsilon}(\Omega)}.
    \end{equation*}
    On the other hand, using Theorem \ref{thmregdiri} and Lemma \ref{lemtheestimateofHblablainH}, we obtain
    \begin{equation*}
    \begin{aligned}
    &\|w\|_{\hat{H}^{s+1/2-\epsilon}(\Omega)}\leq\|Aw\|_{\hat{H}^{1/2-s-\epsilon}(\Omega)}
    \\
    &
    \leq C\|(g(z)+A-g(z))(g(z)+A)^{-1}v\|_{\hat{H}^{1/2-s-\epsilon}(\Omega)}\\
    &\leq C\|v\|_{\hat{H}^{1/2-s-\epsilon}(\Omega)}+C|g(z)|\|w\|_{\hat{H}^{1/2-s-\epsilon}(\Omega)}\leq C\|v\|_{\hat{H}^{1/2-s-\epsilon}(\Omega)}.
    \end{aligned}
    \end{equation*}
    Thus
    \begin{equation*}
    \begin{aligned}
    & |g(z)|\|e\|^2_{L^2(\Omega)}+\|e\|^2_{\hat{H}^s(\Omega)}
    \\
    & \leq Ch^{1/2-\epsilon}\|v\|_{\hat{H}^{1/2-s-\epsilon}(\Omega)}\left(|g(z)|^{1/2}\|e\|_{L^2(\Omega)}+\|e\|_{\hat{H}^s(\Omega)}\right),
    \end{aligned}
    \end{equation*}
    which leads to
    \begin{equation*}
    |g(z)|^{1/2}\|e\|_{L^2(\Omega)}+\|e\|_{\hat{H}^{s}(\Omega)}\leq Ch^{1/2-\epsilon}\|v\|_{\hat{H}^{1/2-s-\epsilon}(\Omega)}.
    \end{equation*}
    Similarly, for $\phi\in L^2(\Omega)$ we set
    \begin{equation*}
    \psi=(g(z)+A)^{-1}\phi, \quad \psi_h=(g(z)+A_h)^{-1}P_h\phi.
    \end{equation*}
    By a duality argument, one has
    \begin{equation*}
    \|e\|_{L^2(\Omega)}= \sup_{\phi\in L^2(\Omega)}\frac{|(e,\phi)|}{\|\phi\|_{L^2(\Omega)}}=\sup_{\phi\in L^2(\Omega)}\frac{|g(z)(e,\psi)+\langle e,\psi\rangle_s|}{\|\phi\|_{L^2(\Omega)}}.
    \end{equation*}
    Then
    \begin{equation*}
    \begin{aligned}
    |g(z)(e,\psi) +\langle e,\psi\rangle_s|=& |g(z)(e,\psi - \psi_h) + \langle e,(\psi - \psi_h )\rangle_s|\\
    \leq&|g(z)|^{1/2}\|e\|_{L^2(\Omega)}|g(z)|^{1/2}\|\psi -\psi_h\|_{L^2(\Omega)}\\
    &+\|e\|_{\hat{H}^s(\Omega)}\|\psi - \psi_h\|_{\hat{H}^s(\Omega)}\\
    \leq& ch^{s+1/2-\epsilon} \|v\|_{\hat{H}^{1/2-s-\epsilon}(\Omega)}\|\phi\|_{L^2(\Omega)},
    \end{aligned}
    \end{equation*}
    where we have used the fact that $|g(z)|^{1/2}\|\psi-\psi_h\|_{L^2(\Omega)}+\|\psi-\psi_h\|_{\hat{H}^s(\Omega)}\leq Ch^{s} \|\phi\|_{L^2(\Omega)}$\cite{Acosta20174}. Thus
    \begin{equation*}
    \|w-w_h\|_{L^2(\Omega)}+h^s\|w-w_h\|_{\hat{H}^s(\Omega)}\leq Ch^{s+1/2-\epsilon}\|v\|_{\hat{H}^{1/2-s-\epsilon}(\Omega)}.
    \end{equation*}
    Combining Lemma \ref{lemeroper} and using interpolation property, one can get the desired estimate.
\end{proof}

For \eqref{equrqtosol}, we give the error estimates for the space semidiscrete scheme with nonsmooth initial values.
\begin{theorem}\label{thmnonsmoothdatasemi}
    Let $G_{1}$, $G_{2}$ and $G_{1,h}$, $G_{2,h}$ be the solutions of the systems \eqref{equrqtosol} and \eqref{equequsysspatialsemiAh}, respectively, $G_{1,0},\ G_{2,0}\in L^2(\Omega)$ and $G_{1,h}(0)=P_hG_{1,0}$, $G_{2,h}(0)=P_hG_{2,0}$. Then
    \begin{equation*}
    \begin{aligned}
    \|G_1(t)-G_{1,h}(t)\|_{L^2(\Omega)}\leq& Ch^{2\gamma_1}(t^{-\alpha_1}\|G_{1,0}\|_{L^2(\Omega)}+t^{\min(0,\alpha_2-\alpha_1)}\|G_{2,0}\|_{L^2(\Omega)})\\
    &+Ch^{2\gamma_2}(\|G_{1,0}\|_{L^2(\Omega)}+\|G_{2,0}\|_{L^2(\Omega)}),\\
    \|G_2(t)-G_{2,h}(t)\|_{L^2(\Omega)}\leq& Ch^{2\gamma_1}(\|G_{1,0}\|_{L^2(\Omega)}+\|G_{2,0}\|_{L^2(\Omega)})\\
    &+Ch^{2\gamma_2}(t^{\min(0,\alpha_1-\alpha_2)}\|G_{1,0}\|_{L^2(\Omega)}+t^{-\alpha_2}\|G_{2,0}\|_{L^2(\Omega)}),\\
    \end{aligned}
    \end{equation*}
    where $\gamma_1\leq \min(s_1,1/2-\epsilon)$ and $\gamma_2\leq \min(s_2,1/2-\epsilon)$ with $\epsilon>0$ being arbitrarily small.
\end{theorem}
\begin{proof}
    From \eqref{equequsysspatialsemilapform}, one can get
    \begin{equation*}
    \begin{aligned}
    \tilde{G}_{1,h}=&H(z,A_{1,h},\alpha_1,\alpha_1-1)P_hG_{1,0}+aH(z,A_{1,h},\alpha_1,\alpha_1-\alpha_2)\tilde{G}_{2,h},\\
    \tilde{G}_{2,h}=&H(z,A_{2,h},\alpha_2,\alpha_2-1)P_hG_{2,0}+aH(z,A_{2,h},\alpha_2,\alpha_2-\alpha_1)\tilde{G}_{1,h}.\\
    \end{aligned}
    \end{equation*}
    Denote $e_1(t)=G_{1}(t)-G_{1,h}(t)$ and $e_2(t)=G_{2}(t)-G_{2,h}(t)$. Combining the above equation with \eqref{equsolrepinlap0} leads to
    \begin{equation}\label{equerrorrep}
    \begin{aligned}
    \tilde{e}_{1}=&z^{\alpha_1-1}(H(z,A_{1},\alpha_1,0)-H(z,A_{1,h},\alpha_1,0)P_h)G_{1,0}\\
    &+aH(z,A_{1},\alpha_1,\alpha_1-\alpha_2)\tilde{G}_{2}-aH(z,A_{1,h},\alpha_1,\alpha_1-\alpha_2)\tilde{G}_{2,h}\\\
    =&z^{\alpha_1-1}(H(z,A_{1,h},\alpha_1,0)-H(z,A_{1,h},\alpha_1,0)P_h)G_{1,0}\\
    &+z^{\alpha_1-\alpha_2}(aH(z,A_{1},\alpha_1,0)\tilde{G}_{2}-aH(z,A_{1,h},\alpha_1,0)P_h\tilde{G}_{2})\\
    &+aH(z,A_{1,h},\alpha_1,\alpha_1-\alpha_2)P_h\tilde{G}_{2}-aH(z,A_{1,h},\alpha_1,\alpha_1-\alpha_2)\tilde{G}_{2,h}\\
    =&\sum_{i=1}^{3}\tilde{\uppercase\expandafter{\romannumeral1}}_i,\\
    \tilde{e}_{2}=&z^{\alpha_2-1}(H(z,A_{2},\alpha_2,0)-H(z,A_{2,h},\alpha_2,0)P_h)G_{2,0}\\
    &+aH(z,A_{2},\alpha_2,\alpha_2-\alpha_1)\tilde{G}_{1}-aH(z,A_{2,h},\alpha_2,\alpha_2-\alpha_1)\tilde{G}_{1,h}\\\
    =&z^{\alpha_2-1}(H(z,A_{2,h},\alpha_2,0)-H(z,A_{2,h},\alpha_2,0)P_h)G_{2,0}\\
    &+z^{\alpha_2-\alpha_1}(aH(z,A_{2},\alpha_2,0)\tilde{G}_{1}-aH(z,A_{2,h},\alpha_2,0)P_h\tilde{G}_{1})\\
    &+aH(z,A_{2,h},\alpha_2,\alpha_2-\alpha_1)P_h\tilde{G}_{1}-aH(z,A_{2,h},\alpha_2,\alpha_2-\alpha_1)\tilde{G}_{1,h}\\
    =&\sum_{i=1}^{3}\tilde{\uppercase\expandafter{\romannumeral2}}_i.
    \end{aligned}
    \end{equation}
    We first consider $e_1$. For $\uppercase\expandafter{\romannumeral1}_1$, using inverse Laplace transform and Lemma \ref{lemeroper} leads to
    \begin{equation*}
    \begin{aligned}
    \|\uppercase\expandafter{\romannumeral1}_1\|_{L^2(\Omega)}\leq Ch^{2\gamma_1}\int_{\Gamma_{\theta,\kappa}}|e^{zt}||z|^{\alpha_1-1}|dz|\|G_{1,0}\|_{L^2(\Omega)}\leq Ch^{2\gamma_1}t^{-\alpha_1}\|G_{1,0}\|_{L^2(\Omega)}.
    \end{aligned}
    \end{equation*}
    For $\uppercase\expandafter{\romannumeral1}_2$, taking inverse Laplace transform, using Eq. \eqref{equsolrepinlap0}, Lemma \ref{lemeroper}, and Theorem \ref{thmhomoregularitynonsmooth}, we have
    \begin{equation*}
    \begin{aligned}
    \|\uppercase\expandafter{\romannumeral1}_2\|_{L^2(\Omega)}\leq
    &Ch^{2\gamma_1}(\|G_{1,0}\|_{L^2(\Omega)}+t^{\min(0,\alpha_2-\alpha_1)}\|G_{2,0}\|_{L^2(\Omega)}),
    \end{aligned}
    \end{equation*}
    where we have used the fact $T/t\leq1$.
    As for $ \uppercase\expandafter{\romannumeral1}_3 $, similar to Lemma \ref{lemtheestimateofHblabla}, one has
    \begin{equation*}
    \|H(z,A_{i,h},\alpha,\beta)\|_{L^{2}(\Omega)\rightarrow L^{2}(\Omega)}\leq C|z|^{\beta-\alpha}.
    \end{equation*}
    Then the inverse Laplace transform and the $L_2$ stability of projection $P_h$ lead to
    \begin{equation*}
    \|\uppercase\expandafter{\romannumeral1}_3\|_{L^2(\Omega)}\leq C\int_{0}^t(t-s)^{\alpha_2-1}\|e_2(s)\|_{L^2(\Omega)}ds.
    \end{equation*}
    Thus
    \begin{equation*}
    \begin{aligned}
    \|e_1(t)\|_{L^2(\Omega)}\leq & Ch^{2\gamma_1}(t^{-\alpha_1}\|G_{1,0}\|_{L^2(\Omega)}+t^{\min(0,\alpha_2-\alpha_1)}\|G_{2,0}\|_{L^2(\Omega)})
    \\
    &+C\int_{0}^t(t-s)^{\alpha_2-1}\|e_2(s)\|_{L^2(\Omega)}ds.
    \end{aligned}
    \end{equation*}
    Similarly, there also exists
        \begin{equation*}
    \begin{aligned}
    \|e_2(t)\|_{L^2(\Omega)}\leq& Ch^{2\gamma_2}(t^{-\alpha_2}\|G_{2,0}\|_{L^2(\Omega)}+t^{\min(0,\alpha_1-\alpha_2)}\|G_{1,0}\|_{L^2(\Omega)})
        \\&+C\int_{0}^t(t-s)^{\alpha_1-1}\|e_1(s)\|_{L^2(\Omega)}ds.
    \end{aligned}
    \end{equation*}
Thus, the desired estimates can be obtained by Lemma \ref{lemgrondwall} and the fact $T/t>1$.
\end{proof}

Finally, combining the proof of Theorem \ref{thmnonsmoothdatasemi}, the priori estimate provided in Section \ref{Sec1}, and Lemma \ref{lemeroper2}, there are the following spatial error estimates for $s_1<1/2$. Theorems \ref{thmsmoothdatasemi1}, \ref{thmsmoothdatasemi2}, and \ref{thmsmoothdatasemi3}, are with different assumptions on the regularities of the initial values and/or the range of $s_2$.
\begin{theorem}\label{thmsmoothdatasemi1}
     Let $G_{1}$, $G_{2}$ and $G_{1,h}$, $G_{2,h}$ be the solutions of the systems \eqref{equrqtosol} and \eqref{equequsysspatialsemiAh}, respectively. Assume $s_1\leq s_2<1/2$, $G_{1,0},\ G_{2,0}\in \hat{H}^\sigma(\Omega)$, $\sigma< \max(1/2-s_1,1/2-s_2)$ and $G_{1,h}(0)=P_hG_{1,0}$, $G_{2,h}(0)=P_hG_{2,0}$. Then
    \begin{equation*}
    \begin{aligned}
    &\|G_1(t)-G_{1,h}(t)\|_{L^2(\Omega)}
    \\
    &
    \leq Ch^{s_1+\gamma_1}\left(t^{-\alpha_1}\|G_{1,0}\|_{\hat{H}^\sigma(\Omega)}+t^{\min(0,\alpha_2-\alpha_1)}\|G_{2,0}\|_{\hat{H}^\sigma(\Omega)}\right)\\
    &+Ch^{s_2+\gamma_2}\left(\|G_{1,0}\|_{\hat{H}^\sigma(\Omega)}+\|G_{2,0}\|_{\hat{H}^\sigma(\Omega)}\right),\\
    & \|G_2(t)-G_{2,h}(t)\|_{L^2(\Omega)}
    \\ &
    \leq
     Ch^{s_1+\gamma_1}\left(\|G_{1,0}\|_{\hat{H}^\sigma(\Omega)}+\|G_{2,0}\|_{\hat{H}^\sigma(\Omega)}\right)\\
    &+Ch^{s_2+\gamma_2}\left(t^{\min(0,\alpha_1-\alpha_2)}\|G_{1,0}\|_{\hat{H}^\sigma(\Omega)}+t^{-\alpha_2}\|G_{2,0}\|_{\hat{H}^\sigma(\Omega)}\right),
    \end{aligned}
    \end{equation*}
    where $\gamma_1\leq \min(s_1+\sigma,1/2-\epsilon)$ and $\gamma_2\leq \min(s_2+\sigma,1/2-\epsilon)$ with $\epsilon>0$ being arbitrarily small.
\end{theorem}
\begin{theorem}\label{thmsmoothdatasemi2}
    Let $G_{1}$, $G_{2}$ and $G_{1,h}$, $G_{2,h}$ be the solutions of the systems \eqref{equrqtosol} and \eqref{equequsysspatialsemiAh}, respectively.  Assume $s_1\leq s_2<1/2$, $G_{i,0}\in \hat{H}^{\sigma_i}(\Omega)$, $\sigma_i<1/2-\epsilon$ $(i=1,2)$ and $G_{1,h}(0)=P_hG_{1,0}$, $G_{2,h}(0)=P_hG_{2,0}$. Denote $\mu_1=\max(\frac{\alpha_1-\alpha_2}{\alpha_1}+\epsilon,0)$, $\mu_2=\max(\frac{\alpha_2-\alpha_1}{\alpha_2}+\epsilon,0)$ and $\bar{\gamma}_i=\min(1/2-\epsilon,s_i+\sigma_i)$ $(i=1,2)$.
    \begin{itemize}
    \item If $\sigma_1+2\mu_1s_1\leq2s_2+\sigma_2$ and $\sigma_2+2\mu_2s_2\leq2s_1+\sigma_1$, then
    \begin{equation*}
    \begin{aligned}
   & \|G_1(t)-G_{1,h}(t)\|_{L^2(\Omega)}
    \\
    &
    \leq Ch^{s_1+\gamma_1}\left(t^{-\alpha_1}\|G_{1,0}\|_{\hat{H}^{\sigma_1}(\Omega)}+t^{\min(0,\alpha_2-\alpha_1)}\|G_{2,0}\|_{\hat{H}^{\sigma_2}(\Omega)}\right)\\
    &+Ch^{s_2+\gamma_2}\left(\|G_{1,0}\|_{\hat{H}^{\sigma_1}(\Omega)}+\|G_{2,0}\|_{\hat{H}^{\sigma_2}(\Omega)}\right),\\
    & \|G_2(t)-G_{2,h}(t)\|_{L^2(\Omega)}\leq\\
    & Ch^{s_1+\gamma_1}\left(\|G_{1,0}\|_{\hat{H}^{\sigma_1}(\Omega)}+\|G_{2,0}\|_{\hat{H}^{\sigma_2}(\Omega)}\right)\\
    &+Ch^{s_2+\gamma_2}\left(t^{\min(0,\alpha_1-\alpha_2)}\|G_{1,0}\|_{\hat{H}^{\sigma_1}(\Omega)}+t^{-\alpha_2}\|G_{2,0}\|_{\hat{H}^{\sigma_2}(\Omega)}\right).\\
    \end{aligned}
    \end{equation*}
    \item Assume $\sigma_1>\sigma_2$. If $\sigma_1+2\mu_1s_1>2s_2+\sigma_2$ or $\sigma_2+2\mu_2s_2>2s_1+\sigma_1$, then
    \begin{equation*}
    \begin{aligned}
    & \|G_1(t)-G_{1,h}(t)\|_{L^2(\Omega)}
    \\
    &
    \leq Ch^{s_1+\gamma_1}\left(t^{-\alpha_1}\|G_{1,0}\|_{\hat{H}^{\sigma_1}(\Omega)}+t^{\min(0,\alpha_2-\alpha_1)}\|G_{2,0}\|_{\hat{H}^{\sigma_2}(\Omega)}\right)\\
    &+Ch^{s_2+\bar{\gamma}_2}\left(\|G_{1,0}\|_{\hat{H}^{\sigma_1}(\Omega)}+\|G_{2,0}\|_{\hat{H}^{\sigma_2}(\Omega)}\right),\\
    & \|G_2(t)-G_{2,h}(t)\|_{L^2(\Omega)}\\
    &\leq Ch^{s_1+\gamma_1}\left(\|G_{1,0}\|_{\hat{H}^{\sigma_1}(\Omega)}+\|G_{2,0}\|_{\hat{H}^{\sigma_2}(\Omega)}\right)\\
    &+Ch^{s_2+\bar{\gamma}_2}\left(t^{\min(0,\alpha_1-\alpha_2)}\|G_{1,0}\|_{\hat{H}^{\sigma_1}(\Omega)}+t^{-\alpha_2}\|G_{2,0}\|_{\hat{H}^{\sigma_2}(\Omega)}\right).\\
    \end{aligned}
    \end{equation*}
    \item Assume $\sigma_1<\sigma_2$. If $\sigma_1+2\mu_1s_1>2s_2+\sigma_2$ or $\sigma_2+2\mu_2s_2>2s_1+\sigma_1$, then
    \begin{equation*}
    \begin{aligned}
   & \|G_1(t)-G_{1,h}(t)\|_{L^2(\Omega)}
    \\
    &
    \leq Ch^{s_1+\bar{\gamma}_1}\left(t^{-\alpha_1}\|G_{1,0}\|_{\hat{H}^{\sigma_1}(\Omega)}+t^{\min(0,\alpha_2-\alpha_1)}\|G_{2,0}\|_{\hat{H}^{\sigma_2}(\Omega)}\right)\\
    &+Ch^{s_2+\gamma_2}\left(\|G_{1,0}\|_{\hat{H}^{\sigma_1}(\Omega)}+\|G_{2,0}\|_{\hat{H}^{\sigma_2}(\Omega)}\right),\\
   & \|G_2(t)-G_{2,h}(t)\|_{L^2(\Omega)}
   \\
    &\leq Ch^{s_1+\bar{\gamma}_1}\left(\|G_{1,0}\|_{\hat{H}^{\sigma_1}(\Omega)}+\|G_{2,0}\|_{\hat{H}^{\sigma_2}(\Omega)}\right)\\
    &+Ch^{s_2+\gamma_2}\left(t^{\min(0,\alpha_1-\alpha_2)}\|G_{1,0}\|_{\hat{H}^{\sigma_1}(\Omega)}+t^{-\alpha_2}\|G_{2,0}\|_{\hat{H}^{\sigma_2}(\Omega)}\right).\\
    \end{aligned}
    \end{equation*}
    \end{itemize}
    Here $\gamma_1= \min(s_1+\sigma_1,s_1+s_2+\bar{\gamma}_2-2\mu_1s_1,1/2-\epsilon)$ and $\gamma_2= \min(s_2+\sigma_2,s_2+s_1+\bar{\gamma}_1-2\mu_2s_2,1/2-\epsilon)$ with $\epsilon>0$ being arbitrarily small.
\end{theorem}

\begin{theorem}\label{thmsmoothdatasemi3}
    Let $G_{1}$, $G_{2}$ and $G_{1,h}$, $G_{2,h}$ be the solutions of the systems \eqref{equrqtosol} and \eqref{equequsysspatialsemiAh}, respectively. Assume $s_1<1/2\leq s_2$, $G_{i,0}\in \hat{H}^{\sigma_i}(\Omega)$ $(i=1,2)$, $\sigma_1<1/2-s_1$, $\sigma_2=0$ and $G_{1,h}(0)=P_hG_{1,0}$, $G_{2,h}(0)=P_hG_{2,0}$. Then
    \begin{equation*}
    \begin{aligned}
    & \|G_1(t)-G_{1,h}(t)\|_{L^2(\Omega)}
    \\
    &
    \leq Ch^{s_1+\gamma_1}\left(t^{-\alpha_1}\|G_{1,0}\|_{\hat{H}^{\sigma_1}(\Omega)}+t^{\min(0,\alpha_2-\alpha_1)}\|G_{2,0}\|_{\hat{H}^{\sigma_2}(\Omega)}\right)\\
    &+Ch^{2\gamma_2}\left(\|G_{1,0}\|_{\hat{H}^{\sigma_1}(\Omega)}+\|G_{2,0}\|_{\hat{H}^{\sigma_2}(\Omega)}\right),\\
    & \|G_2(t)-G_{2,h}(t)\|_{L^2(\Omega)}
    \\
    &
    \leq Ch^{s_1+\gamma_1}\left(\|G_{1,0}\|_{\hat{H}^{\sigma_1}(\Omega)}+\|G_{2,0}\|_{\hat{H}^{\sigma_2}(\Omega)}\right)\\
    &+Ch^{2\gamma_2}\left(t^{\min(0,\alpha_1-\alpha_2)}\|G_{1,0}\|_{\hat{H}^{\sigma_1}(\Omega)}+t^{-\alpha_2}\|G_{2,0}\|_{\hat{H}^{\sigma_2}(\Omega)}\right),\\
    \end{aligned}
    \end{equation*}
    where $\gamma_i\leq \min(s_i+\sigma_i,1/2-\epsilon)$, $i=1,2$ with $\epsilon>0$ arbitrary small.
\end{theorem}
\begin{remark}
 From the numerical experiments, we find that the errors aroused by $aH(z,A_1,\alpha_1,\alpha_1-\alpha_2)P_h\tilde{G}_2-aH(z,A_{1,h},\alpha_1,\alpha_1-\alpha_2)\tilde{G}_{2,h}$ and $aH(z,A_2,\alpha_2,\alpha_2-\alpha_1)P_h\tilde{G}_1-aH(z,A_{2,h},\alpha_2,\alpha_2-\alpha_1)\tilde{G}_{1,h}$ in \eqref{equerrorrep} have almost no effect on convergence rates.
\end{remark}
\begin{remark}
As for Eq. \eqref{tramod}, the spatial semidiscrete scheme can be written as
\begin{equation*}
\left (\frac{\partial G_h}{\partial t},v_{h}\right )+~_0D^{1-\alpha}_t\langle G_{h}, v_{h}\rangle_{s}=0,
\end{equation*}
where $v_h\in X_h$ and $G_h(0)=P_h G_0$. According to Lemma \ref{lemeroper2}, if $s<1/2$ and $G_0\in \hat{H}^\sigma(\Omega)$, the error between $G(t)$ and $G_h(t)$ can be written as
\begin{equation*}
\|G(t)-G_h(t)\|_{L^2(\Omega)}+h^s\|G(t)-G_h(t)\|_{\hat{H}^s(\Omega)}\leq C t^{-\alpha}h^{s+\gamma}\|G_0\|_{\hat{H}^{\sigma}(\Omega)},
\end{equation*}
where $\gamma=\min(1/2-\epsilon,s+\sigma)$. And according to Lemma \ref{lemeroper}, if $s\geq1/2$ and $G_0\in L^2(\Omega)$, the error between $G(t)$ and $G_h(t)$ is as follows
\begin{equation*}
\|G(t)-G_h(t)\|_{L^2(\Omega)}+h^s\|G(t)-G_h(t)\|_{\hat{H}^s(\Omega)}\leq C t^{-\alpha}h\|G_0\|_{L^2(\Omega)}.
\end{equation*}
\end{remark}

\section{Time discretization and error analysis}\label{Sec3}
In this section, we use the $L_1$ scheme to discretize the Riemann-Liouville time fractional derivatives and perform the error analysis for the fully discrete scheme.
We first introduce the notations as
\begin{equation}\label{defH1H2}
\begin{aligned}
&H_1(z_1,z_2,A_1,A_2)=((1+az_1+z_1A_1)(1+az_2+z_2A_2)-a^2z_1z_2)^{-1},\\
&H_2(z_1,z_2,A_1,A_2)=H_1(z_1,z_2,A_1,A_2)(1+az_1+z_1A_1).\\
\end{aligned}
\end{equation}
\begin{lemma}[\cite{Nie2018}]\label{lemtheestimateofH1H2}
    When $z\in\Sigma_{\theta,\kappa}$, $ \pi/2<\theta<\pi $ and $\kappa>\max\left (2|a|^{1/\alpha},2|a|^{1/\beta}\right )$, there are the estimates
    \begin{equation*}
    \begin{aligned}
    &\| H_1(z^{-\alpha},z^{-\beta},A_1,A_2)\|\leq C,\  \| H_2(z^{-\alpha},z^{-\beta},A_1,A_2)\|\leq C.
    \end{aligned}
    \end{equation*}

\end{lemma}
According to \eqref{defH1H2}, the solution of \eqref{equequsysspatialsemiAh} in Laplace space can be reconstructed as
\begin{equation}\label{equsemisolrep}
\begin{aligned}
&\tilde{G}_{1,h}=
z^{-1}H_2(z^{-\alpha_2},z^{-\alpha_1},A_{2,h},A_{1,h})G_{1,h}(0)\\
&\quad\quad\quad+aH_1(z^{-\alpha_2},z^{-\alpha_1},A_{2,h},A_{1,h})z^{-1-\alpha_2}G_{2,h}(0),\\
&\tilde{G}_{2,h}=aH_1(z^{-\alpha_1},z^{-\alpha_2},A_{1,h},A_{2,h})z^{-1-\alpha_1}G_{1,h}(0)\\
&\quad\quad\quad+z^{-1}H_2(z^{-\alpha_1},z^{-\alpha_2},A_{1,h},A_{2,h})G_{2,h}(0).
\end{aligned}
\end{equation}
Next, we use the Backward Euler scheme to discretize $\partial /\partial t$ and $L_1$ scheme to approximate $~_0D^{\alpha}_t$. Let the time step size $\tau=T/L$, $L\in\mathbb{N}$, $t_i=i\tau$, $i=0,1,\ldots,L$ and $0=t_0<t_1<\cdots<t_L=T$. Recall the approximation of Caputo fractional derivative by $L_1$ scheme (see, e.g., \cite{Jin2015})
\begin{equation*}
~^C_0D^\alpha_tu(t_n)=\tau^{-\alpha}\left (b^\alpha_0u(t_n)+\sum_{i=1}^{n-1}(b^\alpha_j-b^\alpha_{j-1})u(t_{n-j})-b^\alpha_{n-1}u(t_0)\right )+\mathcal{O}(\tau^{2-\alpha}),
\end{equation*}
where
\begin{equation*}
b^\alpha_j=((j+1)^{1-\alpha}-j^{1-\alpha})/\Gamma(2-\alpha),~j=0,1,\cdots,n-1.
\end{equation*}
Using the relationship between the Caputo fractional derivative and the Riemann-Liouville fractional derivative, i.e.,
\begin{equation*}
~_0D^\alpha_tu(t)=~^C_0D^\alpha_tu(t)+\frac{t^{-\alpha}}{\Gamma(1-\alpha)}u(0),
\end{equation*}
we obtain
\begin{equation*}
~_0D^\alpha_tu(t_n)=\tau^{-\alpha}\sum_{i=0}^{n}d^\alpha_ju(t_{n-j})+\mathcal{O}(\tau^{2-\alpha}),
\end{equation*}
where
\begin{equation}\label{equweightdalpha}
d^\alpha_j=\left\{
\begin{aligned}
&b^\alpha_0\qquad\qquad\qquad\qquad \: \:  {\rm for}~j=0,\\
&b^\alpha_j-b^\alpha_{j-1}\qquad\qquad\quad \: {\rm for}~0<j<n,\\
&b^\alpha_{n-1}+\frac{n^{-\alpha}}{\Gamma(1-\alpha)}\qquad {\rm for}~j=n.
\end{aligned}\right.
\end{equation}
For the system \eqref{equrqtosol}, we have the fully discrete scheme
\begin{equation}\label{equfulldis}
\left \{\begin{aligned}
&\frac{G^n_{1,h}-G^{n-1}_{1,h}}{\tau}+a\tau^{\alpha_1-1}\sum_{i=0}^{n-1}d^{1-\alpha_1}_iG^{n-i}_{1,h}+
\tau^{\alpha_1-1}\sum_{i=0}^{n-1}d^{1-\alpha_1}_iA_{1,h} G^{n-i}_{1,h}
\\
&
\quad\quad=a\tau^{\alpha_2-1}\sum_{i=0}^{n-1}d^{1-\alpha_2}_iG^{n-i}_{2,h},\\
&\frac{G^n_{2,h}-G^{n-1}_{2,h}}{\tau}+a\tau^{\alpha_2-1}\sum_{i=0}^{n-1}d^{1-\alpha_2}_iG^{n-i}_{2,h}+
\tau^{\alpha_2-1}\sum_{i=0}^{n-1}d^{1-\alpha_2}_iA_{2,h} G^{n-i}_{2,h}
\\
&
\quad\quad
=a\tau^{\alpha_1-1}\sum_{i=0}^{n-1}d^{1-\alpha_1}_iG^{n-i}_{1,h},\\
&G^0_{1,h}=G_{1,h}(0),\\
&G^0_{2,h}=G_{2,h}(0),
\end{aligned}\right .
\end{equation}
where $ G^n_{1,h} $, $G^n_{2,h}$ are the numerical solutions of  $ G_{1} $, $G_{2}$ at time $t_n$.

To get the error estimate between \eqref{equrqtosol} and \eqref{equfulldis}, we introduce $Li_p(z)$ \cite{Flajolet1999,Jin2015} defined by
\begin{equation*}
Li_p(z)=\sum_{j=1}^{\infty}\frac{z^j}{j^p},
\end{equation*}
and recall the Lemmas about $Li_p(z)$ .
\begin{lemma}[\cite{Flajolet1999,Jin2015}]
    For $p\neq 1,2,\cdots$, the function $Li_p(e^{-z})$ satisfies the singular expansion
    \begin{equation*}
    Li_p(e^{-z})\sim \Gamma(1-p)z^{p-1}+\sum_{l=0}^{\infty}(-1)^l\varsigma(p-l)\frac{z^l}{l!} \qquad as~z\rightarrow 0,
    \end{equation*}
    where $\varsigma(z)$ denotes the Riemann zeta function.
\end{lemma}
\begin{lemma}[\cite{Flajolet1999,Jin2015}]\label{lemLipabconvergence}
    Let $|z|\leq \frac{\pi}{\sin(\theta)}$ with $\theta\in (\frac{\pi}{2},\frac{5\pi}{6})$ and $-1<p<0$. Then
    \begin{equation*}
    Li_p(e^{-z})=\Gamma(1-p)z^{p-1}+\sum_{l=0}^{\infty}(-1)^l\varsigma(p-l)\frac{z^l}{l!}
    \end{equation*}
    converges absolutely.
\end{lemma}

At the same time, we have the estimates.
\begin{lemma}[\cite{Jin2015,Jin2016}]\label{lemestzz}
    Assume $z\in\Sigma_{\theta}$, $|z|\leq\frac{\pi}{\tau\sin(\theta)}$ and $\theta\in(\pi/2,\pi)$. Then there are
    \begin{equation*}
    \begin{aligned}
    C_1|z|\leq\left |\frac{1-e^{-z\tau}}{\tau}\right |\leq C_2|z|,\qquad \left |\frac{1-e^{-z\tau}}{\tau}-z\right |\leq C\tau|z|^2.
    \end{aligned}
    \end{equation*}
\end{lemma}
Next, we give the error estimates of the fully discrete scheme. To get the solutions of the system \eqref{equfulldis}, multiplying $\zeta^n$ on both sides of the first two equations in (\ref{equfulldis}), summing $n$ from $1$ to $\infty$  and using \eqref{equweightdalpha}, there exist
\begin{equation}\label{equnumsollapformtoest1}
\begin{aligned}
\sum_{i=1}^{\infty}G^i_{1,h}\zeta^i=&\frac{\zeta}{\tau}\left (\frac{1-\zeta}{\tau}\right )^{-1}\Bigg(H_{2}(\psi^{1-\alpha_2}(\zeta),\psi^{1-\alpha_1}(\zeta),A_{2,h},A_{1,h})G_{1,h}(0)\\
&+aH_1(\psi^{1-\alpha_2}(\zeta),\psi^{1-\alpha_1}(\zeta),A_{2,h},A_{1,h})\psi^{1-\alpha_2}(\zeta)G_{2,h}(0)\Huge \Bigg),\\
\end{aligned}
\end{equation}
\begin{equation}\label{equnumsollapformtoest2}
\begin{aligned}
&\sum_{i=1}^{\infty}G^i_{2,h}\zeta^i
\\
&=\frac{\zeta}{\tau}\left (\frac{1-\zeta}{\tau}\right )^{-1}\Bigg (aH_1(\psi^{1-\alpha_1}(\zeta),\psi^{1-\alpha_2}(\zeta),A_{1,h},A_{2,h})\psi^{1-\alpha_1}(\zeta)G_{1,h}(0)\\
&\quad+H_{2}(\psi^{1-\alpha_1}(\zeta),\psi^{1-\alpha_2}(\zeta),A_{1,h},A_{2,h})G_{2,h}(0)\Bigg ),
\end{aligned}
\end{equation}
where
\begin{equation*}
\psi^{\alpha}(\zeta)=\tau^{-\alpha}\left (\frac{1-\zeta}{\tau}\right )^{-1}\left( \sum_{j=0}^{\infty}d^{\alpha}_j\zeta^j\right ).
\end{equation*}
As for $\psi^{\alpha}(\zeta)$, using the definition of $d^\alpha_j$ and $Li_p(z)$,
we have
\begin{equation*}
\begin{aligned}
\psi^{\alpha}(\zeta)=&\tau^{-\alpha}\left (\frac{1-\zeta}{\tau}\right )^{-1}\left( \sum_{j=1}^{\infty}(b^{\alpha}_j-b^{\alpha}_{j-1})\zeta^j+b^\alpha_0\zeta^0\right )\\
=&\tau^{1-\alpha} \sum_{j=0}^{\infty}b^{\alpha}_j\zeta^j=\frac{\tau^{1-\alpha}}{\Gamma(2-\alpha)} \left(\sum_{j=0}^{\infty}((j+1)^{1-\alpha}-j^{1-\alpha})\zeta^j\right)\\
=&\frac{\tau^{1-\alpha}}{\Gamma(2-\alpha)}\frac{(1-\zeta)}{\zeta} \left(\sum_{j=0}^{\infty}j^{1-\alpha}\zeta^j\right)=\frac{\tau^{1-\alpha}}{\Gamma(2-\alpha)}\frac{(1-\zeta)}{\zeta} Li_{\alpha-1}(\zeta).
\end{aligned}
\end{equation*}
Then, there is the following estimate.
\begin{lemma}\label{lemestpsi}
    Let $z\in \Gamma_{\theta,\kappa}$, $ |z\tau|\leq \frac{\pi}{sin(\theta)}$ and $ \theta\in (\pi/2,5\pi/6) $. Then we have
    \begin{equation*}
    \left |\psi^{\alpha}(e^{-z\tau})-z^{\alpha-1}\right |\leq C\tau|z|^{\alpha}.
    \end{equation*}
\end{lemma}
\begin{proof}
    By Lemma \ref{lemLipabconvergence}, there exists
    \begin{equation*}
    \begin{aligned}
    \psi^{\alpha}(e^{-z\tau})=&\tau^{1-\alpha}\sum_{j=1}^{\infty}\frac{(z\tau)^j}{j!}\left [(z\tau)^{\alpha-2}+\sum_{k=0}^{\infty}\frac{(-1)^k\varsigma(-\alpha-k)}{\Gamma(2-\alpha)}\frac{(z\tau)^k}{k!}\right ]\\
    =&z^{\alpha-1}+\sum_{j=2}^{\infty}\frac{z^{\alpha-2+j}\tau^{j-1}}{j!}+\tau^{1-\alpha}\sum_{j=1}^{\infty}\frac{(z\tau)^j}{j!}\sum_{k=0}^{\infty}\frac{(-1)^k\varsigma(-\alpha-k)}{\Gamma(2-\alpha)}\frac{(z\tau)^k}{k!}\\
    =&z^{\alpha-1}+\mathcal{O}(|z|^{\alpha}\tau).
    \end{aligned}
    \end{equation*}
\end{proof}

Now we give the error estimates between the solutions of the systems \eqref{equequsysspatialsemiAh} and \eqref{equfulldis}.
\begin{theorem}\label{thmhomfullest}
    Let $G_{1,h}$, $G_{2,h}$ and $G^n_{1,h}$, $G^n_{2,h}$ be the solutions of the systems \eqref{equequsysspatialsemiAh} and \eqref{equfulldis}, respectively. Then
    \begin{equation*}
    \begin{aligned}
    \|G_{1,h}(t_n)-G^n_{1,h}\|_{L^2(\Omega)}\leq C\tau(t_n^{-1}\|G_{1,h}(0)\|_{L^2(\Omega)}+t_n^{\alpha_2-1}\|G_{2,h}(0)\|_{L^2(\Omega)}),\\
    \|G_{2,h}(t_n)-G^n_{2,h}\|_{L^2(\Omega)}\leq C\tau(t_n^{\alpha_1-1}\|G_{1,h}(0)\|_{L^2(\Omega)}+t_n^{-1}\|G_{2,h}(0)\|_{L^2(\Omega)}).
    \end{aligned}
    \end{equation*}
\end{theorem}
\begin{proof}
    We first consider the error estimates between $ G^n_{1,h} $ and $ G_{1,h} $. By \eqref{equnumsollapformtoest1}, for small $\xi_\tau=e^{-\tau(\kappa+1)}$, there is
    \begin{equation*}
    \begin{aligned}
    G^n_{1,h}=&\frac{1}{2\pi \mathbf{i}\tau}\int_{|\zeta|=\xi_\tau}\zeta^{-n-1}\zeta\left (\frac{1-\zeta}{\tau}\right )^{-1}\Bigg(H_{2}(\psi^{1-\alpha_2}(\zeta),\psi^{1-\alpha_1}(\zeta),A_{2,h},A_{1,h})\\
    &\cdot G_{1,h}(0)+aH_1(\psi^{1-\alpha_2}(\zeta),\psi^{1-\alpha_1}(\zeta),A_{2,h},A_{1,h})\psi^{1-\alpha_2}(\zeta)G_{2,h}(0)\Huge \Bigg) d\zeta.
    \end{aligned}
    \end{equation*}
Letting $\zeta=e^{-z\tau}$ leads to
    \begin{equation*}
    \begin{aligned}
    G^n_{1,h}=&\frac{1}{2\pi \mathbf{i}}\int_{\Gamma^\tau}e^{zt_n}e^{-z\tau}\left (\frac{1-e^{-z\tau}}{\tau}\right )^{-1}
    \\
    &
    \cdot\Bigg(H_{2}(\psi^{1-\alpha_2}(e^{-z\tau}),\psi^{1-\alpha_1}(e^{-z\tau}),A_{2,h},A_{1,h})G_{1,h}(0)\\
    &+aH_1(\psi^{1-\alpha_2}(e^{-z\tau}),\psi^{1-\alpha_1}(e^{-z\tau}),A_{2,h},A_{1,h})\psi^{1-\alpha_2}(e^{-z\tau})G_{2,h}(0)\Huge \Bigg) dz,
    \end{aligned}
    \end{equation*}
    where $\Gamma^\tau=\{z=\kappa+1+\mathbf{i}y:y\in\mathbb{R}~{\rm and}~|y|\leq \pi/\tau\}$. Next we deform the contour $\Gamma^\tau$ to
    $\Gamma^\tau_{\theta,\kappa}=\{z\in \mathbb{C}:\kappa\leq |z|\leq\frac{\pi}{\tau\sin(\theta)},|\arg z|=\theta\}\cup\{z\in \mathbb{C}:|z|=\kappa,|\arg z|\leq\theta\}$. Thus
    \begin{equation}\label{equfulldissolG1}
    \begin{aligned}
    G^n_{1,h}=& \frac{1}{2\pi \mathbf{i}}\int_{\Gamma^\tau_{\theta,\kappa}}e^{zt_n}e^{-z\tau}\left (\frac{1-e^{-z\tau}}{\tau}\right )^{-1}
    \\
    &
    \cdot \Bigg(H_{2}(\psi^{1-\alpha_2}(e^{-z\tau}),\psi^{1-\alpha_1}(e^{-z\tau}),A_{2,h},A_{1,h})G_{1,h}(0)\\
    &+aH_1(\psi^{1-\alpha_2}(e^{-z\tau}),\psi^{1-\alpha_1}(e^{-z\tau}),A_{2,h},A_{1,h})\psi^{1-\alpha_2}(e^{-z\tau})G_{2,h}(0)\Huge \Bigg)dz.
    \end{aligned}
    \end{equation}
    In view of \eqref{equsemisolrep}, there exists
    \begin{equation}\label{eqG1}
    \begin{aligned}
    G_{1,h}(t)=&\frac{1}{2\pi \mathbf{i}}\int_{\Gamma_{\theta,\kappa}}e^{zt}z^{-1}H_2(z^{-\alpha_2},z^{-\alpha_1},A_{2,h},A_{1,h})G_{1,h}(0)dz\\
    &+\frac{1}{2\pi \mathbf{i}}\int_{\Gamma_{\theta,\kappa}}e^{zt}az^{-1}H_1(z^{-\alpha_2},z^{-\alpha_1},A_{2,h},A_{1,h})z^{-\alpha_2}G_{2,h}(0)dz.
    \end{aligned}
    \end{equation}
    Combining \eqref{equfulldissolG1} and \eqref{eqG1} leads to
    \begin{equation*}
    \begin{aligned}
    &G_{1,h}(t_n)-G^n_{1,h}\\=&\frac{1}{2\pi \mathbf{i}}\int_{\Gamma_{\theta,\kappa}\backslash\Gamma^{\tau}_{\theta,\kappa}}e^{zt_n}z^{-1}H_2(z^{-\alpha_2},z^{-\alpha_1},A_{2,h},A_{1,h})G_{1,h}(0)dz\\
    &+\frac{1}{2\pi \mathbf{i}}\int_{\Gamma_{\theta,\kappa}\backslash\Gamma^{\tau}_{\theta,\kappa}}e^{zt_n}az^{-1}H_1(z^{-\alpha_2},z^{-\alpha_1},A_{2,h},A_{1,h})z^{-\alpha_2}G_{2,h}(0)dz\\
    &+\frac{1}{2\pi \mathbf{i}}\int_{\Gamma^{\tau}_{\theta,\kappa}}e^{zt_n}\Bigg (z^{-1}H_2(z^{-\alpha_2},z^{-\alpha_1},A_{2,h},A_{1,h})-e^{-z\tau}\left (\frac{1-e^{-z\tau}}{\tau}\right )^{-1}\\
    &\cdot H_{2}(\psi^{1-\alpha_2}(e^{-z\tau}),\psi^{1-\alpha_1}(e^{-z\tau}),A_{2,h},A_{1,h})\Bigg )G_{1,h}(0)dz\\
    &+\frac{a}{2\pi \mathbf{i}}\int_{\Gamma^{\tau}_{\theta,\kappa}}e^{zt_n}\Bigg(z^{-1}H_1(z^{-\alpha_2},z^{-\alpha_1},A_{2,h},A_{1,h})z^{-\alpha_2}-e^{-z\tau}\left (\frac{1-e^{-z\tau}}{\tau}\right )^{-1}\\
    &\cdot H_1(\psi^{1-\alpha_2}(e^{-z\tau}),\psi^{1-\alpha_1}(e^{-z\tau}),A_{2,h},A_{1,h})\psi^{1-\alpha_2}(e^{-z\tau})\Bigg)G_{2,h}(0)dz\\
    =&\uppercase\expandafter{\romannumeral1}+\uppercase\expandafter{\romannumeral2}+\uppercase\expandafter{\romannumeral3}+\uppercase\expandafter{\romannumeral4}.
    \end{aligned}
    \end{equation*}
    According to Lemma \ref{lemtheestimateofH1H2}, there exists
    \begin{equation*}
    \begin{aligned}
    &\|\uppercase\expandafter{\romannumeral1}\|_{L^2(\Omega)}
    \\
    &\leq C\int_{\Gamma_{\theta,\kappa}\backslash\Gamma^{\tau}_{\theta,\kappa}}e^{-C|z|t_n}|z|^{-1} \|H_{2}(z^{-\alpha_2},z^{-\alpha_1},A_{2,h},A_{1,h})\||dz|\|G_{1,h}(0)\|_{L^2(\Omega)}\\&\leq C t_n^{-1}\tau\|G_{1,h}(0)\|_{L^2(\Omega)}.
    \end{aligned}
    \end{equation*}
    For $ \uppercase\expandafter{\romannumeral2} $, similarly it has
    \begin{equation*}
    \begin{aligned}
    &\|\uppercase\expandafter{\romannumeral2}\|_{L^2(\Omega)}
    \\
    &\leq C\int_{\Gamma_{\theta,\kappa}\backslash\Gamma^\tau_{\theta,\kappa}}e^{-C|z|t_n}|z|^{-\alpha_2-1}\|a H_1(z^{-\alpha_2},z^{-\alpha_1},A_{2,h},A_{1,h})\||dz|\|G_{2,h}(0)\|_{L^2(\Omega)}\\&\leq C t_n^{\alpha_2-1}\tau\|G_{2,h}(0)\|_{L^2(\Omega)}.
    \end{aligned}
    \end{equation*}
    Next for $ \uppercase\expandafter{\romannumeral3} $ and $ \uppercase\expandafter{\romannumeral4} $, there are
    \begin{equation*}
    \begin{aligned}
    &\uppercase\expandafter{\romannumeral3}
    \\
    &=\frac{1}{2\pi \mathbf{i}}\int_{\Gamma^{\tau}_{\theta,\kappa}}e^{zt_n}e^{-z\tau}\Bigg (z^{-1}H_{2}(z^{-\alpha_2},z^{-\alpha_1},A_{2,h},A_{1,h})-\\&\left (\frac{1-e^{-z\tau}}{\tau}\right )^{-1}H_{2}(z^{-\alpha_2},z^{-\alpha_1},A_{2,h},A_{1,h})\Bigg )G_{1,h}(0)dz\\
    &+\frac{1}{2\pi \mathbf{i}}\int_{\Gamma^{\tau}_{\theta,\kappa}}e^{zt_n}e^{-z\tau}\Bigg (\left (\frac{1-e^{-z\tau}}{\tau}\right )^{-1}H_{2}(z^{-\alpha_2},z^{-\alpha_1},A_{2,h},A_{1,h})\\
    &-\left (\frac{1-e^{-z\tau}}{\tau}\right )^{-1}H_{2}(\psi^{1-\alpha_2}(e^{-z\tau}),\psi^{1-\alpha_1}(e^{-z\tau}),A_{2,h},A_{1,h})\Bigg )G_{1,h}(0)dz\\
    &+\frac{1}{2\pi \mathbf{i}}\int_{\Gamma^{\tau}_{\theta,\kappa}}e^{zt_n}e^{-z\tau}(e^{z\tau}-1)z^{-1}H_{2}(z^{-\alpha_2},z^{-\alpha_1},A_{2,h},A_{1,h})dz
    \\
    &
    =\uppercase\expandafter{\romannumeral3}_1+\uppercase\expandafter{\romannumeral3}_2+\uppercase\expandafter{\romannumeral3}_3
    \end{aligned}
    \end{equation*}
    and
    \begin{equation*}
    \begin{aligned}
   & \uppercase\expandafter{\romannumeral4} \\
    &=\frac{a}{2\pi \mathbf{i}}\int_{\Gamma^{\tau}_{\theta,\kappa}}e^{zt_n}e^{-z\tau}\Bigg(z^{-1}H_{1}(z^{-\alpha_2},z^{-\alpha_1},A_{2,h},A_{1,h})z^{-\alpha_2}\\
    &- \left (\frac{1-e^{-z\tau}}{\tau}\right )^{-1}H_{1}(z^{-\alpha_2},z^{-\alpha_1},A_{2,h},A_{1,h})z^{-\alpha_2}\Bigg)G_{2,h}(0)dz\\
    &+\frac{a}{2\pi \mathbf{i}}\int_{\Gamma^{\tau}_{\theta,\kappa}}e^{zt_n}e^{-z\tau}\Bigg(\left (\frac{1-e^{-z\tau}}{\tau}\right )^{-1}H_{1}(z^{-\alpha_2},z^{-\alpha_1},A_{2,h},A_{1,h})z^{-\alpha_2}\\
    &- \left (\frac{1-e^{-z\tau}}{\tau}\right )^{-1}H_{1}(\psi^{1-\alpha_2}(e^{-z\tau}),\psi^{1-\alpha_1}(e^{-z\tau}),A_{2,h},A_{1,h})\psi^{1-\alpha_2}(e^{-z\tau})\Bigg)
    \\
    &
    \cdot G_{2,h}(0)dz\\
    &+\frac{a}{2\pi \mathbf{i}}\int_{\Gamma^{\tau}_{\theta,\kappa}}e^{zt_n}e^{-z\tau}(e^{z\tau}-1)z^{-1}H_{1}(z^{-\alpha_2},z^{-\alpha_1},A_{2,h},A_{1,h})z^{-\alpha_2}dz
    \\
    &
    =\uppercase\expandafter{\romannumeral4}_1+\uppercase\expandafter{\romannumeral4}_2+\uppercase\expandafter{\romannumeral4}_3.
    \end{aligned}
    \end{equation*}
As for $\uppercase\expandafter{\romannumeral3}_1$ and $\uppercase\expandafter{\romannumeral4}_1$, using Lemmas \ref{lemtheestimateofH1H2} and \ref{lemestzz} leads to
    \begin{equation*}
    \begin{aligned}
    &\left \|z^{-1}H_{2}(z^{-\alpha_2},z^{-\alpha_1},A_{2,h},A_{1,h})-\right.\\&\quad\left.\left (\frac{1-e^{-z\tau}}{\tau}\right )^{-1}H_{2}(z^{-\alpha_2},z^{-\alpha_1},A_{2,h},A_{1,h})\right \|_{L^2(\Omega)\rightarrow L^2(\Omega)}\leq C\tau
    \end{aligned}
    \end{equation*}
    and
    \begin{equation*}
    \begin{aligned}
    &\left \|z^{-1}H_1(z^{-\alpha_2},z^{-\alpha_1},A_{2,h},A_{1,h})z^{-\alpha_2}-\right.\\&\quad\left. \left (\frac{1-e^{-z\tau}}{\tau}\right )^{-1}H_1(z^{-\alpha_2},z^{-\alpha_1},A_{2,h},A_{1,h})z^{-\alpha_2}\right \|_{L^2(\Omega)\rightarrow L^2(\Omega)}\leq C\tau|z|^{-\alpha_2}.
    \end{aligned}
    \end{equation*}
    Thus
    \begin{equation*}
    \|\uppercase\expandafter{\romannumeral3}_1\|_{L^2(\Omega)}\leq C\tau\int_{\Gamma^{\tau}_{\theta,\kappa}}e^{-C|z|t_{n-1}}|dz|\|G_{1,h}(0)\|_{L^2(\Omega)}\leq C t_n^{-1}\tau\|G_{1,h}(0)\|_{L^2(\Omega)}
    \end{equation*}
    and
    \begin{equation*}
    \begin{aligned}
    &
    \|\uppercase\expandafter{\romannumeral4}_1\|_{L^2(\Omega)}
    \\
    &
    \leq C\tau\int_{\Gamma^{\tau}_{\theta,\kappa}}e^{-C|z|t_{n-1}}|z|^{-\alpha_2}|dz|\|G_{2,h}(0)\|_{L^2(\Omega)}\leq C t_n^{\alpha_2-1}\tau\|G_{2,h}(0)\|_{L^2(\Omega)}.
    \end{aligned}
    \end{equation*}
Denote $H^{(a,b)}(z_1,z_2,A_1,A_2)$ as the $a$-th order derivative of $z_1$ and $b$-th order derivative of $z_2$. Using
 \begin{align*} &\|H^{(1,0)}_{2}(z^{-\alpha_2},z^{-\alpha_1},A_{2,h},A_{1,h})\|_{L^2(\Omega)\rightarrow L^2(\Omega)}\leq
    C|z|^{\alpha_2} ,\\ &\|H^{(0,1)}_{2}(z^{-\alpha_2},z^{-\alpha_1},A_{2,h},A_{1,h})\|_{L^2(\Omega)\rightarrow L^2(\Omega)}\leq
    C|z|^{\alpha_1} ,\\
    &\|H^{(0,1)}_{1}(z^{-\alpha_2},z^{-\alpha_1},A_{2,h},A_{1,h})z^{-\alpha_2}\|_{L^2(\Omega)\rightarrow L^2(\Omega)}\leq
    C|z|^{-\alpha_1-\alpha_2},\\ &\|H^{(1,0)}_{1}(z^{-\alpha_2},z^{-\alpha_1},A_{2,h},A_{1,h})z^{-\alpha_2}-H_{1}(z^{-\alpha_2},z^{-\alpha_1},A_{2,h},A_{1,h})\|_{L^2(\Omega)\rightarrow L^2(\Omega)}\\
    &
    \leq
    C,
    \end{align*}
    the mean value theorem, and the Lemmas \ref{lemestzz} and \ref{lemestpsi}, there are
    \begin{equation*}
    \begin{aligned}
    &\left \|\left (\frac{1-e^{-z\tau}}{\tau}\right )^{-1}H_{2}(z^{-\alpha_2},z^{-\alpha_1},A_{2,h},A_{1,h})-\right.\\
    &\left.\left (\frac{1-e^{-z\tau}}{\tau}\right )^{-1}H_{2}(\psi^{1-\alpha_2}(e^{-z\tau}),\psi^{1-\alpha_1}(e^{-z\tau}),A_{2,h},A_{1,h})\right \|_{L^2(\Omega)\rightarrow L^2(\Omega)}\leq C\tau
    \end{aligned}
    \end{equation*}
    and
    \begin{equation*}
    \begin{aligned}
    &\Bigg \|\left (\frac{1-e^{-z\tau}}{\tau}\right )^{-1}H_{1}(z^{-\alpha_2},z^{-\alpha_1},A_{2,h},A_{1,h})z^{-\alpha_2}-\left (\frac{1-e^{-z\tau}}{\tau}\right )^{-1}\\
    & \cdot H_{1}(\psi^{1-\alpha_2}(e^{-z\tau}),\psi^{1-\alpha_1}(e^{-z\tau}),A_{2,h},A_{1,h})\psi^{1-\alpha_2}(e^{-z\tau})\Bigg \|_{L^2(\Omega)\rightarrow L^2(\Omega)}
    \\
    &
    \leq C\tau|z|^{-\alpha_2}.
    \end{aligned}
    \end{equation*}
    Thus
    \begin{equation*}
    \|\uppercase\expandafter{\romannumeral3}_2\|_{L^2(\Omega)}\leq C\tau\int_{\Gamma^{\tau}_{\theta,\kappa}}e^{-C|z|t_{n-1}}|dz|\|G_{1,h}(0)\|_{L^2(\Omega)}\leq C t_n^{-1}\tau\|G_{1,h}(0)\|_{L^2(\Omega)}
    \end{equation*}
    and
    \begin{equation*}
        \begin{aligned}
    \|\uppercase\expandafter{\romannumeral4}_2\|_{L^2(\Omega)}&\leq C\tau\int_{\Gamma^{\tau}_{\theta,\kappa}}e^{-C|z|t_{n-1}}|z|^{-\alpha_2}|dz|\|G_{2,h}(0)\|_{L^2(\Omega)}
    \\
    & \leq C t_n^{\alpha_2-1}\tau\|G_{2,h}(0)\|_{L^2(\Omega)}.
        \end{aligned}
    \end{equation*}
    By simple calculations, we have
           \begin{equation*}
    \|\uppercase\expandafter{\romannumeral3}_3\|_{L^2(\Omega)}\leq C\tau\int_{\Gamma^{\tau}_{\theta,\kappa}}e^{-C|z|t_{n-1}} |dz|\|G_{1,h}(0)\|_{L^2(\Omega)}\leq Ct_n^{-1}\tau\|G_{1,h}(0)\|_{L^2(\Omega)}
        \end{equation*}
       \begin{equation*}
    \begin{aligned}
      \|\uppercase\expandafter{\romannumeral4}_3\|_{L^2(\Omega)} &\leq C\tau\int_{\Gamma^{\tau}_{\theta,\kappa}}e^{-C|z|t_{n-1}} |z|^{-\alpha_2}|dz|\|G_{2,h}(0)\|_{L^2(\Omega)}
      \\
      &
      \leq Ct_n^{\alpha_2-1}\tau\|G_{2,h}(0)\|_{L^2(\Omega)}.
    \end{aligned}
    \end{equation*}
    In summary,
    \begin{equation*}
    \|G_{1,h}(t_n)-G^n_{1,h}\|_{L^2(\Omega)}\leq C\tau\left(t_n^{-1}\|G_{1,h}(0)\|_{L^2(\Omega)}+t_n^{\alpha_2-1}\|G_{2,h}(0)\|_{L^2(\Omega)}\right).
    \end{equation*}
    Analogously, it has
    \begin{equation*}
    \|G_{2,h}(t_n)-G^n_{2,h}\|_{L^2(\Omega)}\leq C\tau\left(t_n^{\alpha_1-1}\|G_{1,h}(0)\|_{L^2(\Omega)}+t_n^{-1}\|G_{2,h}(0)\|_{L^2(\Omega)}\right).
    \end{equation*}
    The proof has been completed.
\end{proof}

\section{Numerical experiments} \label{Sec4}
In this section, we perform the numerical experiments to verify the effectiveness of the designed schemes. Since the exact solutions $G_1$ and $G_2$ are unknown, to get the spatial convergence rates, we calculate
\begin{equation*}
	\begin{aligned}
		E_{1,h}=\|G^{n}_{1,h}-G^{n}_{1,h/2}\|_{L^2(\Omega)},\quad
		E_{2,h}=\|G^{n}_{2,h}-G^{n}_{2,h/2}\|_{L^2(\Omega)},
	\end{aligned}
\end{equation*}
where $G^n_{1,h}$ and $G^n_{2,h}$ mean the numerical solutions of $G_{1}$ and $G_{2}$ at time $t_n$ with mesh size $h$; similarly, to obtain the temporal convergence rates, we calcuate
\begin{equation*}
	\begin{aligned}
		E_{1,\tau}=\|G_{1,\tau}-G_{1,\tau/2}\|_{L^2(\Omega)},\quad
		E_{2,\tau}=\|G_{2,\tau}-G_{2,\tau/2}\|_{L^2(\Omega)},
	\end{aligned}
\end{equation*}
where $G_{1,\tau}$ and $G_{2,\tau}$ are the numerical solutions of $G_{1}$ and $G_{2}$ at the fixed time $t$ with step size $\tau$. Then the spatial and temporal convergence rates can be, respectively, obtained by
\begin{equation*}
	{\rm Rate}=\frac{\ln(E_{i,h}/E_{i,h/2})}{\ln(2)},\quad {\rm Rate}=\frac{\ln(E_{i,\tau}/E_{i,\tau/2})}{\ln(2)},\quad\quad  i=1,2.
\end{equation*}
The following two groups of initial values are used:
\begin{enumerate}[(a)]
\item
\begin{equation*}\label{iniH05}
	G_{1}(x,0)=\chi_{(1/2,1)},\qquad G_{2}(x,0)=\chi_{(0,1/2)};
	\end{equation*}
    \item
    \begin{equation*}\label{initial_c}
		G_{1}(x,0)=(1-x)^{-\nu_1},\qquad G_{2}(x,0)=x^{-\nu_2},
	\end{equation*}
\end{enumerate}
where $\chi_{(a,b)}$ denotes the characteristic function on $(a,b)$.

Here we first give some examples to show the influence of the regularity of initial data on convergence rates.
\begin{example}

We take $a=2$, $\tau=1/800$, and $T=1$ to solve the system \eqref{equrqtosol} with the initial condition \eqref{iniH05}, and $s_1=s_2<1/2$, $\alpha_1=0.4$, $\alpha_2=0.7$. Here $G_{1,0}, G_{2,0}\in \hat{H}^{1/2-\epsilon}(\Omega)$ satisfy the conditions of Theorem \ref{thmnonsmoothdatasemi}. 
Table \ref{tab:samesdisconls} shows that the convergence rates can be achieved as $O(h^{s+1/2-\epsilon})$,  which agree with Theorems \ref{thmsmoothdatasemi1} and \ref{thmsmoothdatasemi2}.

\begin{table}
\caption{$L_2$ errors and convergence rates with $s_1=s_2=s<1/2$ and the initial condition \eqref{iniH05}}
\label{tab:samesdisconls}
\begin{tabular}{c|c|ccccc}
\hline
      $s$ &     $1/h$ &         50 &        100 &        200 &        400 &        800 \\
\hline
           & $E_{1,h}$ &  1.293E-02 &  8.631E-03 &  5.752E-03 &  3.829E-03 &  2.546E-03 \\

       0.1 &            &       Rate &    0.5834  &    0.5854  &    0.5872  &    0.5888  \\

           & $E_{2,h}$ &  9.139E-03 &  6.038E-03 &  3.986E-03 &  2.630E-03 &  1.734E-03 \\

           &            &       Rate &    0.5981  &    0.5991  &    0.5999  &    0.6005  \\
\hline
           & $E_{1,h}$ &  5.861E-03 &  3.486E-03 &  2.071E-03 &  1.229E-03 &  7.298E-04 \\

      0.25 &            &       Rate &    0.7496  &    0.7514  &    0.7522  &    0.7523  \\

           & $E_{2,h}$ &  3.795E-03 &  2.247E-03 &  1.331E-03 &  7.890E-04 &  4.680E-04 \\

           &            &       Rate &    0.7562  &    0.7553  &    0.7544  &    0.7535  \\
\hline
           & $E_{1,h}$ &  2.334E-03 &  1.247E-03 &  6.681E-04 &  3.585E-04 &  1.925E-04 \\

       0.4 &            &       Rate &    0.9044  &    0.9006  &    0.8982  &    0.8971  \\

           & $E_{2,h}$ &  1.468E-03 &  7.863E-04 &  4.218E-04 &  2.264E-04 &  1.216E-04 \\

           &            &       Rate &    0.9010  &    0.8985  &    0.8974  &    0.8975  \\
\hline
\end{tabular}
\end{table}

\end{example}

\begin{example}\label{Ex2}
We take $\alpha_1=0.4$, $\alpha_2=0.6$, $a=-2$, $\tau=1/800$, and $T=1$ to solve the system \eqref{equrqtosol} with the initial condition \eqref{iniH05}. Table \ref{tab:diffs} shows the $L_2$ errors and convergence rates for different values of $s_1,\,s_2$. The convergence rates are consistent with the results of Theorem \ref{thmnonsmoothdatasemi} when $s_1,s_2>1/2$; when $s_1,s_2<1/2$, the convergence rates of $G_2$ are higher than the predicted ones in Theorem \ref{thmsmoothdatasemi1} (or Theorem \ref{thmsmoothdatasemi2}) and the convergence rates of $G_1$ are the same as the predicted ones, the reason of which may be the less effect of  $aH(z,A_1,\alpha_1,\alpha_1-\alpha_2)P_h\tilde{G}_2-aH(z,A_{1,h},\alpha_1,\alpha_1-\alpha_2)\tilde{G}_{2,h}$ and $aH(z,A_2,\alpha_2,\alpha_2-\alpha_1)P_h\tilde{G}_1-aH(z,A_{2,h},\alpha_2,\alpha_2-\alpha_1)\tilde{G}_{1,h}$  in \eqref{equerrorrep} on convergence rates.
\begin{table}
\caption{$L_2$ errors and convergence rates with different $s_1,s_2$ and the initial condition \eqref{iniH05}}
\label{tab:diffs}
\begin{tabular}{c|c|ccccc}
\hline
$(s_1,s_2)$ &     $1/h$ &         50 &        100 &        200 &        400 &        800 \\
\hline
           & $E_{1,h}$ &  1.173E-02 &  7.797E-03 &  5.181E-03 &  3.441E-03 &  2.284E-03 \\

   (0.1,0.2) &            &          Rate &    0.5894  &    0.5899  &    0.5905  &    0.5913  \\

           & $E_{2,h}$ &  6.456E-03 &  3.988E-03 &  2.460E-03 &  1.516E-03 &  9.340E-04 \\

           &            &          Rate &    0.6949  &    0.6969  &    0.6982  &    0.6992  \\
\hline
           & $E_{1,h}$ &  4.105E-03 &  2.349E-03 &  1.345E-03 &  7.707E-04 &  4.420E-04 \\

   (0.3,0.4) &            &          Rate &    0.8051  &    0.8045  &    0.8034  &    0.8023  \\

           & $E_{2,h}$ &  1.853E-03 &  9.921E-04 &  5.320E-04 &  2.856E-04 &  1.534E-04 \\

           &            &          Rate &    0.9017  &    0.8989  &    0.8973  &    0.8968  \\
\hline
           & $E_{1,h}$ &  5.780E-04 &  2.754E-04 &  1.326E-04 &  6.425E-05 &  3.109E-05 \\

   (0.6,0.7) &            &          Rate &    1.0695  &    1.0547  &    1.0453  &    1.0472  \\

           & $E_{2,h}$ &  2.347E-04 &  1.092E-04 &  5.172E-05 &  2.479E-05 &  1.193E-05 \\

           &            &          Rate &    1.1032  &    1.0785  &    1.0613  &    1.0546  \\
\hline
           & $E_{1,h}$ &  1.143E-04 &  4.905E-05 &  2.194E-05 &  1.013E-05 &  4.798E-06 \\

   (0.8,0.9) &            &          Rate &    1.2207  &    1.1607  &    1.1149  &    1.0780  \\

           & $E_{2,h}$ &  3.297E-05 &  1.330E-05 &  5.799E-06 &  2.668E-06 &  1.269E-06 \\

           &            &          Rate &    1.3098  &    1.1974  &    1.1202  &    1.0714  \\
\hline
\end{tabular}

\end{table}
\end{example}

\begin{example}
The parameters are taken as $\alpha_1=0.8$, $\alpha_2=0.9$, $a=2$, $\tau=1/800$, and $T=1$. First, we solve the system \eqref{equrqtosol} with the initial condition \eqref{initial_c}. Letting $\nu_1=\nu_2=0.4999$ leads to $G_{1,0},G_{2,0}\in L^2(\Omega)$. According to Table \ref{tab:0505}, the convergence rates agree with Theorem \ref{thmnonsmoothdatasemi} when $s_1,\,s_2>1/2$; when $s_1,\,s_2<1/2$, the convergence rates of $G_2$ are higher than the predicted ones in Theorem \ref{thmnonsmoothdatasemi} and the convergence rates of $G_1$ are the same as the predicted ones, the reason of which is the same as that stated in Example \ref{Ex2}.


\begin{table}
\caption{$L_2$ errors and convergence rates with different $s_1,s_2$ and the initial condition \eqref{initial_c}} ($\nu_1=\nu_2=0.4999$)
\label{tab:0505}

    \begin{tabular}{c|c|ccccc}
    \hline
    $(s_1,s_2)$ & $1/h$ & 50    & 100   & 200   & 400   & 800 \\
    \hline
          & $E_{1,h}$ & 5.682E-02 & 4.505E-02 & 3.629E-02 & 2.967E-02 & 2.459E-02 \\
    (0.1,0.2) &       & Rate     & 0.3348  & 0.3121  & 0.2904  & 0.2709  \\
          & $E_{2,h}$ & 4.932E-02 & 3.583E-02 & 2.609E-02 & 1.906E-02 & 1.398E-02 \\
          &       & Rate     & 0.4612  & 0.4578  & 0.4530  & 0.4475  \\
     \hline
          & $E_{1,h}$ & 9.879E-03 & 6.352E-03 & 4.101E-03 & 2.658E-03 & 1.729E-03 \\
    (0.3,0.4) &       & Rate    & 0.6371  & 0.6310  & 0.6255  & 0.6206  \\
          & $E_{2,h}$ & 8.644E-03 & 5.090E-03 & 2.990E-03 & 1.752E-03 & 1.025E-03 \\
          &       & Rate     & 0.7640  & 0.7677  & 0.7709  & 0.7737  \\
     \hline
          & $E_{1,h}$ & 9.208E-04 & 4.679E-04 & 2.370E-04 & 1.197E-04 & 5.972E-05 \\
    (0.6,0.7) &       & Rate    & 0.9766  & 0.9813  & 0.9852  & 1.0033  \\
          & $E_{2,h}$ & 8.295E-04 & 4.080E-04 & 2.017E-04 & 1.001E-04 & 4.950E-05 \\
          &       & Rate     & 1.0236  & 1.0163  & 1.0111  & 1.0158  \\
     \hline
          & $E_{1,h}$ & 1.290E-04 & 5.828E-05 & 2.703E-05 & 1.280E-05 & 6.170E-06 \\
    (0.8,0.9) &       & Rate     & 1.1462  & 1.1083  & 1.0784  & 1.0530  \\
          & $E_{2,h}$ & 9.645E-05 & 4.097E-05 & 1.837E-05 & 8.559E-06 & 4.089E-06 \\
          &       & Rate     & 1.2353  & 1.1569  & 1.1020  & 1.0657  \\
    \hline
    \end{tabular}%

\end{table}

Then we take $\nu_1=-0.4$ and $\nu_2=-0.3$, which may lead to $G_{1,0}\in \hat{H}^{0.1}(\Omega)$ and $G_{1,0}\in \hat{H}^{0.2}(\Omega)$. Table \ref{tab:0403} shows the convergence results and we find the convergence rates of $G_2$ are higher than the predicted ones in Theorem \ref{thmnonsmoothdatasemi} and the convergence rates of $G_1$ are the same as the predicted ones, and the reason for these phenomena is the same as the one in Example \ref{Ex2}.
\begin{table}
\caption{$L_2$ errors and convergence rates with different $s_1,s_2$ and the initial condition \eqref{initial_c}  ($\nu_1=-0.4$, $\nu_2=-0.3$)}
\label{tab:0403}

    \begin{tabular}{c|c|ccccc}
    \hline
    $(s_1,s_2)$ & $1/h$ & 50    & 100   & 200   & 400   & 800 \\
    \hline
          & $E_{1,h}$ & 3.524E-02 & 2.633E-02 & 1.993E-02 & 1.528E-02 & 1.184E-02 \\
    (0.1,0.2) &       & Rate  & 0.4207  & 0.4016  & 0.3837  & 0.3677  \\
          & $E_{2,h}$ & 2.688E-02 & 1.785E-02 & 1.182E-02 & 7.810E-03 & 5.151E-03 \\
          &       & Rate  & 0.5908  & 0.5947  & 0.5978  & 0.6005  \\
     \hline
          & $E_{1,h}$ & 6.941E-03 & 4.271E-03 & 2.631E-03 & 1.622E-03 & 1.001E-03 \\
    (0.3,0.4) &       & Rate  & 0.7006  & 0.6991  & 0.6977  & 0.6965  \\
          & $E_{2,h}$ & 5.211E-03 & 2.880E-03 & 1.586E-03 & 8.700E-04 & 4.759E-04 \\
          &       & Rate  & 0.8556  & 0.8609  & 0.8659  & 0.8705  \\
    \hline
    \end{tabular}%

\end{table}
\end{example}

\begin{example}
In this example, we take $\alpha_1=0.7$, $\alpha_2=0.6$, $a=0.1$, $\tau=1/50$, and $T=20$. The system \eqref{equrqtosol} is solved with the initial condition \eqref{initial_c} and we take $\nu_1=0$, $\nu_2=0.4999$, which implies $G_{1,0}\in \hat{H}^{1/2-\epsilon}(\Omega)$, $G_{2,0}\in L^2(\Omega)$. According to Table \ref{tab:0005}, the results for $s_1=0.25$ and $s_2=0.8$ agree with Theorem \ref{thmsmoothdatasemi3}; when $s_1,\,s_2<1/2$, the convergence rates of $G_1$ are higher than the predicted ones in Theorem \ref{thmsmoothdatasemi2} and the convergence rates of $G_2$ are the same as the predicted ones, the reason of which is the same as that stated in Example \ref{Ex2}.

 \begin{table}
\caption{$L_2$ errors and convergence rates with different $s_1,s_2$ and the initial condition \eqref{initial_c} ($\nu_1=0$, $\nu_2=0.4999$)}
\label{tab:0005}

    \begin{tabular}{c|c|ccccc}
    \hline

    $(s_1,s_2)$ & $1/h$ & 50    & 100   & 200   & 400   & 800 \\
    \hline
          & $E_{1,h}$ & 3.630E-04 & 1.980E-04 & 1.080E-04 & 5.894E-05 & 3.217E-05 \\
    (0.4,0.1) &       & Rate  & 0.8746  & 0.8742  & 0.8739  & 0.8737  \\
          & $E_{2,h}$ & 2.591E-02 & 2.269E-02 & 1.985E-02 & 1.736E-02 & 1.517E-02 \\
          &       & Rate  & 0.1916  & 0.1926  & 0.1936  & 0.1947  \\
    \hline
          & $E_{1,h}$ & 3.417E-04 & 1.849E-04 & 1.001E-04 & 5.411E-05 & 2.924E-05 \\
    (0.4,0.2) &       & Rate  & 0.8858  & 0.8862  & 0.8869  & 0.8882  \\
          & $E_{2,h}$ & 1.122E-02 & 8.629E-03 & 6.622E-03 & 5.072E-03 & 3.878E-03 \\
          &       & Rate  & 0.3784  & 0.3818  & 0.3848  & 0.3874  \\
     \hline
          & $E_{1,h}$ & 8.932E-05 & 4.412E-05 & 2.196E-05 & 1.100E-05 & 5.464E-06 \\
    (0.6,0.3) &       & Rate  & 1.0175  & 1.0065  & 0.9974  & 1.0095  \\
          & $E_{2,h}$ & 4.877E-03 & 3.321E-03 & 2.252E-03 & 1.521E-03 & 1.024E-03 \\
          &       & Rate  & 0.5544  & 0.5606  & 0.5660  & 0.5707  \\
    \hline
    \end{tabular}%

\end{table}
\end{example}
Finally, we verify the temporal convergence rates in the following example.
\begin{example}
Here we take $s_1=0.25$, $s_2=0.75$, $a=2$, and $h=1/400$ to solve the system \eqref{equrqtosol} with the initial condition \eqref{iniH05}. Table \ref{tab:time} shows the $L_2$ errors and convergence rates for different $\alpha_1,\,\alpha_2$, which can be used to validate the results of Theorem \ref{thmhomfullest}.
\begin{table}
\caption{$L_2$ errors and convergence rates with different $\alpha_1$, $\alpha_2$ and the initial condition \eqref{iniH05}}
\label{tab:time}
\begin{tabular}{c|c|ccccc}
\hline
$(\alpha_1,\alpha_2)$ &  $1/\tau$ &        100 &        200 &        400 &        800 &       1600 \\
\hline
           & $E_{1,\tau}$ &  3.980E-02 &  1.957E-02 &  9.745E-03 &  4.876E-03 &  2.444E-03 \\

   (0.3,0.6) &            &       Rate &    1.0241  &    1.0059  &    0.9988  &    0.9966  \\

           & $E_{2,\tau}$ &  1.038E-01 &  5.130E-02 &  2.565E-02 &  1.288E-02 &  6.478E-03 \\

           &            &       Rate &    1.0173  &    0.9999  &    0.9935  &    0.9920  \\
\hline
           & $E_{1,\tau}$ &  1.662E-02 &  8.178E-03 &  4.063E-03 &  2.027E-03 &  1.013E-03 \\

   (0.4,0.7) &            &       Rate &    1.0234  &    1.0092  &    1.0031  &    1.0007  \\

           & $E_{2,\tau}$ &  4.338E-02 &  2.145E-02 &  1.070E-02 &  5.358E-03 &  2.685E-03 \\

           &            &       Rate &    1.0159  &    1.0031  &    0.9982  &    0.9967  \\
\hline
           & $E_{1,\tau}$ &  8.279E-03 &  4.071E-03 &  2.019E-03 &  1.006E-03 &  5.020E-04 \\

   (0.25,0.8) &            &       Rate &    1.0242  &    1.0115  &    1.0055  &    1.0026  \\

           & $E_{2,\tau}$ &  2.198E-02 &  1.086E-02 &  5.410E-03 &  2.703E-03 &  1.352E-03 \\

           &            &       Rate &    1.0167  &    1.0059  &    1.0013  &    0.9996  \\
\hline
\end{tabular}

\end{table}
\end{example}
\section{Conclusion}\label{Sec5}
The power law distributions are widely observed in heterogeneous media, relating to the fields of physics, biology, and social science, etc. This paper focuses on the regularity and numerical methods of the two state model with fractional Laplacians, characterizing the power law properties. The priori estimates are obtained under various different regularity assumptions of initial values and/or different powers of fractional Laplacians. The designed numerical scheme is with finite element approximation for fractional Laplacians and $L_1$ scheme to discretize the time fractional Riemann-Liouville derivative. For the scheme, the complete error analyses are provided, and the extensive numerical experiments are performed to validate their  effectiveness.


\begin{acknowledgements}
We thank Buyang Li for the discussions.
\end{acknowledgements}



\end{document}